\documentclass[twoside]{article}

\usepackage{eurosym}
\usepackage[utf8]{inputenc}
\usepackage{mathrsfs}
\usepackage{bbm} 
\usepackage{bm} 
\usepackage{graphicx}
\usepackage{latexsym}
\usepackage{amsmath}
\usepackage{amsfonts}
\usepackage{amssymb}
\usepackage{enumerate}
\usepackage[usenames]{color}
\usepackage[colorlinks,linkcolor=black,anchorcolor=red,citecolor=red]{hyperref}

\newtheorem{theorem}{Theorem}[section]  %
\newtheorem{lemma}{Lemma}[section]
\newtheorem{corollary}{Corollary}[section]
\newtheorem{proposition}{Proposition}[section]
\newtheorem{example}{Example}
\newtheorem{definition}{Definition}

\newtheorem{remark}{Remark}

\newenvironment{proof}{{\it Proof:\enspace}}{\hfill $\blacksquare$\par}
\newenvironment{proof2}{{\it Proof of Proposition~\ref{Prop. PDE2-1}:\enspace}}{\hfill $\blacksquare$\par}
\newenvironment{proof3}{{\it Proof of Proposition~\ref{Prop. PDE2-2}:\enspace}}{\hfill $\blacksquare$\par}

\DeclareMathOperator{\esssup}{ess\,sup}

\newcommand{\diff}[1]{\operatorname{d}\!{#1}} 
\newcommand{\divi}{\operatorname{div}\!}

\allowdisplaybreaks[4]
\begin{document}

\title{\Large \bf  A Unified Lyapunov Method for ISS of PDEs: A Tutorial on   Constructing  Generalized Lyapunov Functionals for Parabolic and Hyperbolic Equations\thanks{
\hfil\break}}

\author{Jun Zheng \& Guchuan Zhu}
  
\date{}
\maketitle

\begin{abstract}
This tutorial provides an overview of the generalized Lyapunov method (GLM) for analyzing input-to-state stability (ISS)  of partial differential equations (PDEs). We begin by revisiting the classical Lyapunov method   and the standard ISS-Lyapunov theorem, highlighting their limitations when applied to systems with complex boundary disturbances. In contrast, the GLM, based on the concept of generalized Lyapunov functionals (GLFs) that explicitly depend on the external input, offers greater flexibility and efficiency, particularly for PDEs with Dirichlet-type disturbances. The main objective of this tutorial is to demonstrate how to systematically construct GLFs to establish ISS estimates in  $L^q$ spaces with any $q\in[2,\infty]$ for different PDEs. Specifically, we consider three representative classes of  PDEs: (i) an
$N$-dimensional nonlinear parabolic equation with mixed nonlinear boundary disturbances, (ii) a first order nonlinear hyperbolic equation with boundary disturbances, and (iii) a second order  linear hyperbolic equation, i.e., a wave equation, with boundary damping and disturbances. For each case, we provide step-by-step constructions of appropriate GLFs and derive explicit ISS estimates, illustrating the general applicability  of the GLM. Finally, we discuss open challenges   and future directions, including the systematic construction of GLFs for broader classes of PDEs and their applications in  controller design.

\end{abstract}

 \tableofcontents 
 
 \section{Introduction}
\(\)
  A fundamental tool for characterizing the robustness of control systems in the presence of disturbances is the notion of input-to-state stability (ISS), first introduced by Sontag and his coauthors for finite-dimensional nonlinear systems described by ordinary differential equations in the late 1980s~\cite{Sontag:1989,Sontag:1990}. Early efforts to extend ISS theory to infinite-dimensional settings, particularly systems governed by partial differential equations (PDEs), date back to the late 2000s and early 2010s; see, e.g.,~\cite{Dashkovskiy:2010,Jayawardhana:2008,Mazenc:2011,Dashkovskiy:2013,Argomedo:2013,Prieur:2011,Prieur:2012}.

The notion of ISS captures the sensitivity of a system even in the presence of nonlinear resonances and provides a qualitative characterization of overshoot behavior when external inputs possess finite energy. As such, it offers a robust framework for PDEs, unifying properties such as bounded-input bounded-output stability, small-input small-error behavior under disturbances, and asymptotic stability in the disturbance-free case. Notably, ISS incorporates the influence of the initial state in a manner fully compatible with classical Lyapunov stability theory, while providing a nonlinear generalization of both finite-gain stability and supremum-norm finite-gain stability. This lays the groundwork for a range of analytical and design tools for PDEs, including small-gain theorems and various characterizations of ISS properties.

In the past decade, a considerable effort has been devoted to providing ISS characterizations
of abstract infinite-dimensional systems \cite{Damak:2021,Damak:2022,Dashkovskiy:2013,Dashkovskiy:2013b, Jacob:2018JDE,Jacob:2020,Jayawardhana:2008, Karafyllis:2011, Mironchenko:2016, Mironchenko:2015SIAM, Mironchenko:2016MCRF, Mironchenko:2018} and  establishing various ISS estimates   for specific PDEs  subject to different types of disturbances \cite{Ahmadi:2016, Argomedo:2013, Argomedo:2012, Auriol:2020, Dashkovskiy:2010, Dashkovskiy:2013, Dashkovskiy:2013b, Jacob:2016, Jacob:2018SIAM, Karafyllis:book,Lamare:2018, Logemann:2013,Mazenc:2011,Mironchenko:2016,Mironchenko:2015SIAM, Mironchenko:2016MCRF, Mironchenko:2018,Pisano:2017, Prieur:2012, Tanwani:2017, Teel:1998, Zheng:201801}.
    For PDEs, external disturbances typically appear either in the domain or on the boundary of the system domain. It is well known that establishing ISS for PDEs with distributed in-domain disturbances can be achieved by directly extending methods developed for finite-dimensional systems, such as the Lyapunov method~\cite{Mironchenko:2020SIAM}.
 In contrast, the investigation of ISS for PDEs with boundary disturbances is considerably more challenging \cite{Karafyllis:book,Mironchenko:2020SIAM}. This is due to the fact that such systems involve unbounded input operators in their abstract formulation, which poses a significant obstacle to ISS analysis; see, e.g.,~\cite{Argomedo:2013,Damak:2021,Damak:2022,Mironchenko:2018,Karafyllis:2016a,Karafyllis:2017,Schwenninger:2020,Karafyllis:book,Mironchenko:2020SIAM,Orlov:2020,Mironchenko:2023}.

To address the ISS of PDEs with boundary disturbances, various approaches have been developed in recent years, including:

\begin{enumerate}[(i)]

\item  \label{(ii)} the finite-difference scheme and eigenfunction expansion of solutions, which can be used to establish ISS for PDEs governed by Sturm--Liouville operators~\cite{Karafyllis:2014,Karafyllis:2016,Karafyllis:2016a,Karafyllis:2017};

\item  \label{(i)} the semigroup and admissibility method, which can be used to establish ISS for PDEs governed by linear differential operators and associated with admissible inputs~\cite{Jacob:2018,Jacob:2016,Jacob:2018SIAM,Jacob:2018JDE,Jacob:2020,Schwenninger:2020};

\item  \label{(vi)} the classical Lyapunov method (CLM) based on variations of Sobolev embedding inequalities, which can be used to establish ISS for PDEs with either Robin or Neumann boundary disturbances~\cite{Zheng:201804,Zheng:201803,Zheng:201802}; this approach was first applied to ISS analysis of parabolic PDEs with boundary disturbances;

\item \label{(iv)} the monotonicity method, which can be used to establish ISS for monotone systems such as parabolic PDEs with Dirichlet boundary disturbances~\cite{Mironchenko:2019};

\item  \label{(v)} the De Giorgi iteration, which can be used to establish ISS for parabolic PDEs with either Dirichlet or Robin boundary disturbances~\cite{Zheng:2019TAC,Zheng:201803,Zheng:2022};

\item  \label{(iii)} the Riesz-spectral approach, which can be used to establish ISS for PDEs governed by Riesz-spectral operators~\cite{Lhachemi:201901,Lhachemi:201902};

\item  \label{(viii)} the maximum principle-based approach, which can be used to establish ISS for parabolic PDEs with different types of boundary disturbances~\cite{Zheng:2019CDC,Zheng:2020c,Zheng:2020SCL,Karafyllis:2021ESAIM};

\item \label{(vii)} the approximative Lyapunov method, which can be used to establish ISS or integral-ISS in different norms, particularly in the $L^1$-norm, for parabolic PDEs with various types of boundary disturbances~\cite{Zheng:2021J,Zheng:202112}.


\end{enumerate}

Note that the methods in    (\ref{(ii)}), (\ref{(i)}), and (\ref{(iii)}),  which  rely on linear operator theory, eigenfunction expansions, or spectral analysis, are particularly well suited for one-dimensional ($1$-D) linear PDEs, whereas the other methods can be applied to either linear or nonlinear PDEs defined on higher dimensional domains.
The methods  presented in (\ref{(vi)}) and (\ref{(vii)}) are  effective for establishing global ISS of   PDEs in the absence of Dirichlet boundary disturbances,  however,    they are not readily extendable to PDEs with Dirichlet disturbances. Note also that extending the methods in (\ref{(iv)}), (\ref{(v)}), and (\ref{(viii)}) to   PDEs beyond parabolic equations is extremely challenging.
Overall, owing to the complexity and diversity of PDEs,
each method has its own advantages and limitations. In particular, the applicability of the non-Lyapunov methods  in (\ref{(ii)}), (\ref{(i)}), and (\ref{(iv)})-(\ref{(viii)}) depends heavily on the structure of the PDEs under consideration.
Consequently, developing a unified ISS analysis framework  applicable to a broad class of PDEs remains a significant challenge.

 Given the dominant role of Lyapunov methods in the stability analysis of control systems, the authors have, over the past few years, been dedicated to developing unified tools  for the ISS analysis of PDEs within the Lyapunov  framework. This effort led to the introduction of the notion of
 \begin{enumerate}
   \item[(ix)]  generalized ISS-Lyapunov functional (GISS-LF)   \cite{Zheng:2024}. In particular, it has been successfully applied to establish    global ISS for the Burgers' equation with Dirichlet boundary disturbances \cite{Zheng:2024}, thereby providing a positive answer to  \cite[Open Problem 5.12]{Mironchenko:2020SIAM}.
       \end{enumerate}
      Furthermore,   using the proposed method, a generalized ISS-Lyapunov theorem was established for abstract control systems, and a GISS-LF was explicitly constructed for a class of parabolic PDEs with Dirichlet boundary disturbances in \cite{Zheng:2025}. A key idea underlying this approach is the construction of a positive-semidefinite Lyapunov candidate that explicitly incorporates  external disturbances, thereby offering a new perspective for handling nonhomogeneous Dirichlet boundary terms in the ISS analysis of PDEs.

It is worth noting that an advantage of the application of GISS-LF lies in its capacity to simultaneously accommodate Dirichlet boundary disturbances alongside other types of disturbances, while enabling ISS analysis  of PDEs in a manner directly analogous to the CLM. In particular, this generalized Lyapunov method (GLM) can proceed in a systematic manner by involving solely  integration by parts (or the Green's formula for higher dimensional PDEs) and fundamental differential inequalities such as  the Agmon's, Poincaré's, Jensen's, Young's, and Gronwall's inequalities, thereby circumventing the need to verify admissibility conditions required  in (\ref{(i)}) for input operators. Moreover, it obviates the need for spectral analysis or other sophisticated analytical tools typically employed in (\ref{(ii)}) and (\ref{(iii)}) and hence, brings convenience to the ISS analysis of concrete PDEs.
      Nevertheless, at present, the applicability of the  GLM remains largely confined to parabolic PDEs, especially those defined on $1$-D domains and with linear boundary conditions (see, e.g., \cite{Bi:2025b,Bi:2025TAC,Zheng:2024,Zheng:2025a}), and its extension to other types of PDEs has not yet been reported.

This tutorial aims to illustrate how the  GLM--although alternative approaches may also be applicable--can be utilized to establish  ISS  for a wider range of  PDEs, in particular, the most representative ones, including higher dimensional parabolic PDEs under general boundary conditions, {$1$-D  first order nonlinear hyperbolic PDEs, and $1$-D   second order  linear hyperbolic PDEs, i.e., the wave equations.}

 In the rest of the paper, we first introduce  some basic notations used in this paper.
 Section~\ref{Sec. II} revisits, within the framework of abstract systems, the    classical ISS-Lyapunov functionals (CISS-LFs)  and the corresponding ISS-Lyapunov theorem, illustrating the effectiveness,  as well as limitations,  of CISS-LFs in the ISS analysis of  PDEs with   boundary disturbances. We then recall the concept of  GISS-LFs  and the associated ISS-Lyapunov theorem, and present a proposition that is particularly well-suited for concrete PDEs. To demonstrate the application of the  GLM  in the ISS analysis of  PDEs, we consider three different types of PDEs:

\begin{itemize}

\item in Section~\ref{GLF-parabolic}, for a class of nonlinear parabolic PDEs defined on a higher dimensional domain with mixed nonlinear boundary disturbances, we show how to construct a GISS-LF and establish ISS estimates in arbitrary $L^q$-norms for $q\in[2,\infty]$;
    \item in Section~\ref{GLF-hyperbolic}, for a class of   {$1$-D first order  hyperbolic equations} with boundary disturbances, we show how to construct a GISS-LF and establish both local and global ISS estimates in arbitrary $L^q$-norms for $q\in[2,\infty]$;
        \item in Section~\ref{GLF-wave}, for a class of {wave equations} with boundary damping and disturbances, we show how to apply the GLM to derive ISS estimates in arbitrary $L^q$-norms for $q\in[2,\infty]$.
    \end{itemize}
    In Section~\ref{sec:conclusion}, we provide some concluding remarks,   followed by   appendices that provide  two lemmas and proofs of several properties of a truncation function used throughout the paper.

\paragraph*{Notation} Let $\mathbb{R}:= (-\infty,+\infty) $,  $
 \mathbb{R}_{\geq0}:= [0,+\infty)$, $
 \mathbb{R}_{>0}:= (0,+\infty)$, and $
 \mathbb{R}_{\leq0}:= (-\infty,0]$.


For $p\in[1,\infty)$ and an open  bounded domain $\Omega \subset\mathbb{R}^N (N\geq 1)$,
the norms of a function   $g\in L^p(\Omega )$  and $L^{\infty}(\Omega)$ are defined by
 $\|g\|_{L^p(\Omega)}:=\left(\int_{\Omega}\left|g(y)\right|^p\mathrm{d}y\right)^{\frac{1}{p}}$ and $\|g\|_{L^{\infty}(\Omega)}:= \mathop{\esssup}\limits_{y\in\Omega}\left|g(y)\right|$, respectively.

%

  Let $C(\mathbb{R}_{\geq 0}; \mathbb{R}_{\geq 0} )$ (or $C(\mathbb{R}_{\geq 0}\times \mathbb{R}_{\geq 0}; \mathbb{R}_{\geq 0})$) be the set  of  continuous functions from $ \mathbb{R}_{\geq 0} $ (or $ \mathbb{R}_{\geq 0} \times \mathbb{R}_{\geq 0}$) to $\mathbb{R}_{\geq 0}$.
  For  a linear normed space $Y$,
let  $PC(\mathbb{R}_{\geq 0};Y) $ be the set  of  all piecewise-right continuous functions $f$ from $\mathbb{R}_{\geq 0}$ to $Y$  satisfying $\|f \|_{PC(\mathbb{R}_{\geq 0};Y)}:=\sup_{t\in \mathbb{R}_{\geq 0}}\|f(t)\|_{Y}<\infty$.

 Denote by   $f\circ g$  the composition of the functions $f$ and $g$, i.e., $f\circ g (\cdot):=f(g(\cdot))$.

Let
\begin{align*}
  \mathcal{P}:=&\{\gamma \in C(\mathbb{R}_{\geq 0};\mathbb{R}_{\geq 0})|\gamma(0)=0, \gamma(s)>0,\forall s\in \mathbb{R}_{>0}\},\\
 \mathcal {K}:=&\{\gamma \in \mathcal{P}|\gamma\ \text{is\ strictly\ increasing}\},\\
 \mathcal{K}_{\infty}:=&\{\gamma \in \mathcal{K}| \lim\limits_{s\to+\infty}\gamma(s)=+\infty\},\\
 \mathcal {L}:=&\{\gamma \in C(\mathbb{R}_{\geq 0};\mathbb{R}_{\geq 0})|\gamma\ \text{is\ strictly\ decreasing\ with} \ \lim\limits_{s\to+\infty}\gamma(s)=0\},\\
 \mathcal {K}\mathcal {L}:=& \{\beta\in  C(\mathbb{R}_{\geq 0}\times \mathbb{R}_{\geq 0};\mathbb{R}_{\geq 0})|\beta (\cdot,t) \in \mathcal {K}, \forall t \in \mathbb{R}_{\geq 0};  \ \beta(s,\cdot)\in \mathcal {L}, \forall s \in {{\mathbb{R}_{>0}}}\}.
\end{align*}

\section{Overview on CLM and  GLM}\label{Sec. II}
In this section, we provide a review of GISS-LF and the associated ISS-Lyapunov theorem for control systems under an abstract form. To enable a meaningful comparison, we begin with a brief recap of the essential concepts of  CLM for control systems as well as the construction of (classical) ISS-Lyapunov functionals (ISS-LFs) for PDEs.
\subsection{A brief review of CLM}
In this subsection, we review the concept of control systems, together with the notions of  ISS  and  classical ISS-LFs, and the ISS-Lyapunov theorem.

We begin by defining   forward complete control systems evolving on a Banach space $X$.

\begin{definition}[Control system, \cite{Mironchenko:2020SIAM}]
 Let $(X,\|\cdot\|_{X})$, $(U,\|\cdot\|_{U})$ be Banach spaces and $\mathcal{U}\subset\{g:\mathbb{R}_{\geq 0}\to U\}$ be a normed vector space satisfying the following two axioms:
 \begin{enumerate}
 \item[(i)] Axiom of shift invariance: for all $u\in \mathcal{U}$ and all $\tau\geq 0$ we have $u(\cdot+\tau)\in \mathcal{U}$ with $\|u\|_{\mathcal{U}}\geq \|u(\cdot+\tau)\|_{\mathcal{U}}$;
     \item[(ii)]  Axiom of concatenation: for all $u_1,u_2\in \mathcal{U}$ and for all $t > 0$ the concatenation
of $u_1$ and $u_2$ at time $t$
\begin{align*}
u_1  \underset{t}{\diamond} u_2(\tau):=\left\{\begin{aligned}
		& u_1(\tau), &&\tau\in[0,t],  \\
		& u_2(\tau-t), &&\text{otherwise},
	\end{aligned}\right.
\end{align*}
belongs to $\mathcal{U}$.
\end{enumerate}
Consider a transition map $\phi: D_{\phi}  \rightarrow X$  with $D_{\phi} \subseteq \mathbb{R}_{\geq 0}\times X\times \mathcal{U}$. The triple $\Sigma=(X,\mathcal{U},\phi)$   is called a control system, if  it verifies the following properties:
\begin{enumerate}[(i)]
    \item Identity property: for every $(x,u)\in X\times \mathcal{U}$, it holds that $\phi(0,x,u)=x$;
        \item Causality: for every $(t,x,u)\in \mathbb{R}_{\geq 0}\times X\times \mathcal{U} $, and for every $\tilde{u}\in \mathcal{U}$ satisfying $\tilde{u}(\tau)=u(\tau)$ for all $\tau\in [0,t]$, it holds that $\phi(t,x,u)=\phi(t,x,\tilde{u})$;
    \item Continuity: for every $(x,u)\in X\times \mathcal{U}$, the map $t\mapsto \phi(t,x,u)$ is continuous;
    \item Cocycle property: for every $t,h\in \mathbb{R}_{\geq 0}$ and every $(x,u)\in X\times \mathcal{U}$, it holds that
    \begin{align*}
    \phi(h,\phi(t,x,u),u(t+\cdot))=\phi(t+h,x,u).
        \end{align*}
        \end{enumerate}
\end{definition}

The following definition is concerned with the forward complete control systems.
 \begin{definition}[Forward complete control system, \cite{Mironchenko:2020SIAM}]
   The control system $\Sigma=(X,\mathcal{U},\phi)$   is said to be  forward complete, if $D_{\phi}=\mathbb{R}_{\geq 0}\times X\times \mathcal{U}$, that is,
 for every  $(t,x,u)\in \mathbb{R}_{\geq 0}\times X\times \mathcal{U} $, the value $\phi(t,x,u)\in X$ is well-defined .
\end{definition}


For a real-valued function $g:\mathbb{R}_{\geq 0}\to \mathbb{R}$ define the right-hand upper Dini derivative at $t\in \mathbb{R}_{\geq 0}$ by
\begin{align*}
D^+g(t):=\limsup_{\tau\to 0^+}\frac{g(t+\tau)-g(t)}{\tau}.
\end{align*}

Let $x\in X$ and $W$ be   a real-valued function defined in a neighborhood of $x$. The
Lie derivative   of $W$ at $x$ corresponding to the input $u\in\mathcal{U}$ along the  trajectory of a control system $\Sigma$ is defined by
\begin{align*}
\dot{W}_u(x):=&D^+W(\phi(\cdot,x,u))\Big|_{t=0}
=
\limsup_{t\to 0^+}\frac{1}{t}(W(\phi(t,x,u))-W(x)).
\end{align*}
We also denote $\dot{W}_u(x)$ by $\dot{W} (x)$ for notational simplicity.

The concept of  ISS  for control systems is presented below.

\begin{definition}[ISS, \cite{Mironchenko:2020SIAM}]\label{Def. ISS}
 A control system $\Sigma=(X,\mathcal{U},\phi)$ is said to be  input-to-state
stable (ISS), if there exist  functions $\beta\in \mathcal{K}\mathcal{L},\gamma\in \mathcal{K}$  such that  for all $(t,x,u)\in \mathbb{R}_{\geq 0}\times X\times \mathcal{U}$ it holds that
\begin{align}\label{ISS-form}
\|\phi(t,x,u)\|_{X}\leq \beta\left( \|x\|_{X},t \right) +  \gamma\left( \|u\|_{\mathcal{U} }\right).
\end{align}
  In addition, the control system~$\Sigma=(X,\mathcal{U},\phi)$ is said to be  locally input-to-state
stable (LISS) if there exists a positive constant $R_0$ such that the inequality \eqref{ISS-form} holds true for all $(t,x,u)\in \mathbb{R}_{\geq 0}\times X\times \mathcal{U}$ satisfying $ \|x\|_{X}+ \|u\|_{\mathcal{U} }\leq R_0$.
\end{definition}

 Note that  under the conditions
\begin{itemize}
    \item[\textbf{(H1)}] $\mathcal{U}:=PC(\mathbb{R}_{\geq 0};U)$,
     \item[\textbf{(H2)}]  $\Sigma=(X,\mathcal{U},\phi)$ is a  forward complete control system with  $x  \equiv 0$ being its equilibrium point,
\end{itemize}
   the causality property of a control system provides an  equivalent definition of ISS, which is   convenient to use  for   PDEs.
\begin{proposition}[\cite{Dashkovskiy:2013}]
 Under the conditions \textbf{(H1)} and \textbf{(H2)}, the control system $\Sigma=(X,\mathcal{U},\phi)$ is ISS if and only if
  there exist  functions $\beta\in \mathcal{K}\mathcal{L},\gamma\in \mathcal{K}$  such that  for all $(t,x,u)\in \mathbb{R}_{\geq 0}\times X\times \mathcal{U}$ it holds that
\begin{align*}
\|\phi(t,x,u)\|_{X}\leq \beta\left( \|x\|_{X},t \right) +  \gamma\left( \sup_{s\in(0,t)}\|u(s)\|_{{U} }\right).
\end{align*}
\end{proposition}

\begin{remark}
The similar property (with $\esssup_{ s\in (0,t)}\|u(s)\|_U$ instead of $\sup_{s\in(0,t)}\|u(s)\|_{{U} }$) holds for continuous inputs, i.e., $\mathcal{U} := C(\mathbb{R}_{\geq 0}, U)$, or the class of strongly measurable and essentially bounded inputs, i.e., $\mathcal{U} :=  L^{\infty}(\mathbb{R}_{\geq 0}, U)$.
\end{remark}

In the following, we present the concept of (classical) ISS-LF for control systems.

\begin{definition}[ISS-LF, \cite{Mironchenko:2020SIAM}]\label{Def:ISS-LF}
A continuous functional $V:X\to\mathbb{R}_{\geq 0}$ is called an      ISS Lyapunov functional (ISS-LF) for the control system $\Sigma$, if there exist functions $\alpha,\psi_1,\psi_2\in\mathcal{K}_{\infty}$ and $\rho\in \mathcal{K}$  such that 
\begin{align}\label{coercive condition}
\psi_2(\|x\|_{X})\leq V(x)\leq \psi_1(\| x\|_X),\forall x\in X,
\end{align}
and  for any   $u\in \mathcal{U}$   the Lie derivative of $V$ at $x$ corresponding to  the input $u$ along the   trajectory satisfies
\begin{align}\label{Lie inequality}
\| x\|_X\geq \rho(\| u\|_\mathcal{U})~~\Rightarrow~~ \dot{V}_{u}(x)\leq -\alpha\left(V(x)\right).
\end{align}

\end{definition}


\begin{remark}\label{Remark1}
  \begin{enumerate}[(i)]


\item More specifically, the functional $V$ satisfying conditions \eqref{coercive condition} and \eqref{Lie inequality} is called  an ISS-LF in  an implication form. If  \eqref{Lie inequality} is replaced by
    \begin{align}\label{Lie inequality-1}
 \dot{V}_{u}(x)\leq -\alpha\left(V(x)\right)+\rho(\| u\|_\mathcal{U}),
\end{align}
    then, $V$ is called an  ISS-LF in  a dissipative form. From   \eqref{Lie inequality} and \eqref{Lie inequality-1}, it is easy to see that   an ISS-LF in a dissipative form must be an ISS-LF   in an implication form. Thus, the condition \eqref{Lie inequality} is weaker than \eqref{Lie inequality-1}.
ISS-LF $V$ is positive definite.
  \item  In \eqref{Lie inequality}, the function $\rho$ is called a Lyapunov gain, and  $\alpha$ is called decay rate of the
ISS-LF $V$.

%

  \item[(iii)] It is worth noting that  the   functional $V $  satisfying  \eqref{coercive condition} and \eqref{Lie inequality}  is also called a coercive ISS-LF due to the fact
      \begin{align*}
     \|x\|_{X}\to +\infty~~\Rightarrow~~    V(x)\to +\infty.
      \end{align*}
If \eqref{coercive condition} is replaced by
     \begin{align}\label{non-coercive condition}
0< V(x)\leq \psi_1(\| x\|_X),\forall x\in X\setminus \{0\},
\end{align}
then, $V $ is called a non-coercive  ISS-LF;  see, e.g., \cite{Mironchenko:2020SIAM}. Note that through such  a definition, a coercive ISS-LF is  also a   non-coercive  ISS-LF.
         \end{enumerate}
 \end{remark}

For a forward complete  control system, the (classical) ISS Lyapunov theorem can be formulated as below.
 \begin{theorem}[ISS Lyapunov theorem, {\cite{Mironchenko:2020SIAM}}]\label{Thm}
 Let $\Sigma=(X,\mathcal{U},\phi)$ be a  forward complete  control system
and   $x  \equiv 0$ be
its equilibrium point. If $\Sigma$ admits an     ISS-LF (either in dissipative or implication form), then it is ISS.
\end{theorem}

The ISS Lyapunov theorem characterizes ISS through the existence of an ISS-LF, thereby offering a systematic framework for analyzing how external inputs affect system stability.
  It establishes a bridge between Lyapunov-based stability analysis and the robustness of nonlinear systems with respect to (w.r.t.) external inputs, serving as a foundational tool for controller design and stability verification in nonlinear control theory.

Since our focus is on demonstrating how to perform ISS analysis within the Lyapunov framework, rather than addressing  the existence and regularity of solutions, we assume throughout this paper that
\begin{enumerate}
    \item[\textbf{(H3)}]
all functions involved in PDEs are sufficiently smooth.
 Consequently,  each PDE considered in this paper admits a classical solution, ensuring that all subsequent calculations in deriving ISS estimates are meaningful.
 \end{enumerate}
 Notably,   if weak or generalized solutions are considered without a fine regularity assumption on the functions, the ISS can still  be assessed by using a density argument \cite[Lemma 1.4.2, p. 25]{Mironchenko:2023b}.

In the following, we demonstrate the construction of an ISS-LF for  the heat equation by means of an example.

\begin{example}\label{Exp-2.1}   Consider
the   heat equation    under mixed  boundary conditions:
 \begin{subequations}\label{Chap.2-equ.10}
   \begin{align}
w_t(y,t)=& w_{yy}(y,t)+ {f(y,t)},  (y,t)\in Q_{\infty},\\
w(0,t)=&0 , t\in\mathbb{R}_{>0},\\
w_y(1,t)=&-w(1,t)+\ {d(t)},t\in\mathbb{R}_{>0},\\
 w(y,0)=&w_0(y), y\in(0,1),
\end{align}
\end{subequations}
 where $Q_{\infty}:=(0,1)\times(0,+\infty)$, $f(y,t)$ represents the in-domain disturbances,     $d(y,t)$  represents the Robin boundary disturbances, and $w_0(y)$ represents initial data.

 It is well-known that the disturbance-free (i.e., $f\equiv d\equiv 0$) system~\eqref{Chap.2-equ.10} is exponentially stable in the $L^2$-norm and
  \begin{align}\label{LF-V}
  V(w):={\int_0^1} w^2(y,t) \diff y
  \end{align}
  is a suitable Lyapunov functional for the stability analysis.  It is also clear that system~\eqref{Chap.2-equ.10} is  ISS    in the  $L^2$-norm  w.r.t.    in-domain and Robin boundary disturbances,  namely, there exist $\beta \in \mathcal {K}\mathcal {L}$ and $ \gamma_1,\gamma_2 \in \mathcal {K}$ such that  \begin{align*}
   \|w(\cdot,t)\|_{L^{2}(0,1)}
    \leq& \beta \left(\|w_0\|_{L^{2}(0,1)},t\right)
          +\gamma_1 \left(  \sup_{s\in (0,t)}\|f(\cdot,s) \|_{L^2(0,1)} \right)+\gamma_2 \left(\sup_{s\in (0,t)}|d(s)|\right),\forall t\in \mathbb{R}_{>0};
\end{align*}
see, e.g., \cite[Theorem 5.3, p. 96]{Karafyllis:book}. Furthermore, since only  in-domain and Robin boundary disturbances are involved,     the functional $V$ given by \eqref{LF-V} is   an ISS-LF; see \cite{Zheng:201804}.
  Indeed, let $X:=L^2(0,1)$ and $\mathcal{U}:=PC(\mathbb{R}_{\geq 0};L^2(0,1) )\times PC(\mathbb{R}_{\geq 0}; \mathbb{R})$. By integrating by parts, the Young's inequality with $\varepsilon\in\left(0,2\right]$ (Lemme~\ref{Young's inequality}), and Wirtinger's   inequality (\cite[(2.31), p.~17]{Krstic:book}), it holds that  
\begin{align}
      \dot{V}(w)
         =&   -2  w^2(1,t)+2 w (1,t){d(t)}  - 2 w(0,t)w_y(0,t) -   2\int_0^1  w_x^2\diff y  + 2\int_0^1 wf\diff y  \label{obstacle} \\
     \leq & -2  w^2(1,t) +  \varepsilon w^{2}(1,t)+\frac{1}{\varepsilon}d^2(t)   +0 - \frac{\pi^2}{2}\int_0^1  w ^2\diff y  +   \varepsilon \int_0^1w^2\diff y+  \frac{1}{ \varepsilon}\int_0^1f^2\diff y \notag\\
     \leq  
    & - \left( \frac{\pi^2}{2} -\varepsilon\right)\int_0^1w^2\diff y  + \frac{1}{\varepsilon}  \left(d^2(t)+ \int_0^1  f^2\diff y\right)  \notag\\
  =& -\left( \frac{\pi^2}{2} -\varepsilon\right)V(w) +\frac{1}{\varepsilon}\left(  \|f(\cdot,t)\|_{L^2(0,1)}^2+|d(t)|^2 \right),   \forall t\in \mathbb{R}_{>0}, \label{Chap.2-equ.12}
\end{align}
 which, along with $ \frac{\pi^2}{2} -\varepsilon>0$, implies that the functional $V$ given by \eqref{LF-V} is   an ISS-LF in a  dissipative form and hence, an ISS-LF in an implication form. Thus,  system~\eqref{Chap.2-equ.10} is  ISS    in the  $L^2$-norm w.r.t.    $(f,d)\in \mathcal{U}$.

  Furthermore, applying the Gronwall's inequality (Lemma~\ref{Gronwall}) to \eqref{Chap.2-equ.12} gives 
 \begin{align*}
      V(w(y,t))
           \leq & e^{-\left( \frac{\pi^2}{2} -\varepsilon\right)t}V(w_0) +\frac{1}{\varepsilon}\int_0^te^{-\left( \frac{\pi^2}{2} -\varepsilon\right)(t-s)}\left(  \|f(\cdot,s)\|_{L^2(0,1)}^2+|d(s)|^2 \right)\diff s\notag\\
  \leq &  e^{-\left( \frac{\pi^2}{2} -\varepsilon\right)t}V(w_0) +\frac{1}{\varepsilon\left(\frac{\pi^2}{2} -\varepsilon\right)}   \sup_{s\in (0,t)}\left( \|f(\cdot,s) \|_{L^2(0,1)}^2  + |d(s)|^2\right).
\end{align*}
 Thus, system~\eqref{Chap.2-equ.10} has the ISS estimate for all $t\in \mathbb{R}_{>0}$:
 \begin{align*}
     \|w(\cdot,t)\|_{L^{2}(0,1)}
            \leq &  e^{-\frac{1}{2}\left( \frac{\pi^2}{2} -\varepsilon\right)t}\|w_0\|_{L^{2}(0,1)} +\frac{1}{\sqrt{\varepsilon\!\left(\frac{\pi^2}{2} -\varepsilon\right)}} \! \left(  \sup_{s\in (0,t)}\!\|f(\cdot,s) \|_{L^2(0,1)}   +\! \sup_{s\in (0,t)}\!|d(s)| \!\right)\!.
\end{align*}

  \end{example}

 Example \ref{Exp-2.1} shows that the Lyapunov functional of the disturbance-free system is sufficient for the ISS analysis in the presence of either in-domain disturbances or Robin boundary disturbances. However, when $w(0,t)=g(t)$ with $g(t)$ representing disturbances, namely, a Dirichlet boundary disturbance is present,  the boundary term $w(0,t)w_y(0,t)$ in \eqref{obstacle} cannot be handled. Thus, the application of the CLM meets obstacles.

  \begin{remark}
  Within the framework of abstract systems and given \emph{a priori} properties of the involved differential operators and input operators  such as analyticity of semigroup, spectral distribution, and admissibility conditions, a non-coercive ISS-LF was constructed in Hilbert spaces for certain autonomous linear systems, including the heat equation with Dirichlet boundary disturbances, by using the inverse of the operators  \cite{Jacob:2020}. Nevertheless,
  without verifying  these conditions, no existing work demonstrates how to directly employ the constructed  non-coercive Lyapuonv functional for ISS analysis of a concrete PDE.
  Notably, to date, it remains unknown whether a coercive ISS-Lyapunov functional exists for the heat equation in the presence of Dirichlet boundary disturbances \cite[Open Problem 3.31]{Mironchenko:2020SIAM}. Consequently, ISS analysis based solely on the CLM still encounters substantial challenges and calls for the development of novel techniques.
\end{remark}


\subsection{A brief review of GLM}

To address the limitations of the CLM,  the notion of generalized Lyapunov functionals has been introduced recently  in \cite{Zheng:2024,Zheng:2025}, which   offers more flexibility for constructing Lyapunov candidates in the ISS analysis of PDEs with  boundary disturbances.

\begin{definition}[GISS-LF,\cite{Zheng:2025}]\label{Def:GISS-LF}For a control system~$\Sigma$, if there exist a continuous functional  $V:X\to\mathbb{R}_{\geq 0}$, a nonempty index set $\mathcal{I}$ having finite elements,  and continuous functions $h_i: X\times \mathcal{U}\to X $ with $i\in \mathcal{I}$ that satisfy  the following
  three conditions:
\begin{enumerate}
 \item[(i)]  there  exists   a function $\psi_1\in\mathcal{K}_{\infty}$  such that
 \begin{align}
  0\leq V(x)\leq \psi_1(\|x\|_X),\forall x\in X;\label{condition-2}
\end{align}
\item[(ii)]  there exist  functions $\psi_0,\psi_2 \in\mathcal{K}_{\infty}$ and $\rho_0,\rho_1 \in \mathcal{K}$   such that
\begin{align}\label{k0}
\|h_i(x,u)\|_{X}\leq  \psi_0\left(\|x\|_X\right) +\rho_0\left(\|u\|_{\mathcal{U}}\right),  {\forall i\in \mathcal{I}}
\end{align}
{and  
\begin{align}
  \psi_2(\| x\|_X)\leq \widehat{V}(x,u) +\rho_1(\| u\|_\mathcal{U})
 \label{condition-1}
\end{align}
hold true for all $ (x,u)\in X\times \mathcal{U}$, where  \begin{align}
 \widehat{V}(x,u) := \sum_{i\in \mathcal{I}}   V(h_i(x,u)) ,\forall (x,u)\in X\times \mathcal{U}; \label{Def.Vtilde}
\end{align}}
\item[(iii)]  there exist functions $\alpha\in \mathcal{K}_{\infty}$ and $ \rho_2\in \mathcal{K}$  such that  the   derivative of {the functional $\widehat{V} $}
  along the   trajectory satisfies
\begin{align}
 &\widehat{V}(x,u)\geq \rho_2(\| u\|_\mathcal{U}) ~
 \Rightarrow    ~   \dot{\widehat{V}} (x,u) \leq - \alpha \left(\widehat{V}(x,u)\right) ,\label{condition-3}
\end{align}

\end{enumerate}
 then, $\widehat{V}(x,u)$ is  called a generalized ISS-Lyapunov functional (GISS-LF) for the control system $\Sigma$.
\end{definition}

\begin{remark} Unlike  CISS-LFs, the GISS-LFs  explicitly incorporate  input-dependent terms through $h_i(x,u)$, thereby providing greater flexibility in the selection of Lyapunov candidates for the ISS analysis. In the special case where $h_i(x,u)\equiv x$ for all $i\in\mathcal{I}$, the GISS-LF reduces to the classical one.  The relationship between  $x,u,h_i,V,\widehat{V}$, and $\mathbb{R}$ is illustrated in Fig. \ref{Fig1}.
\end{remark}


 \begin{figure}[htbp]
  \centering
		\includegraphics[scale=0.35]{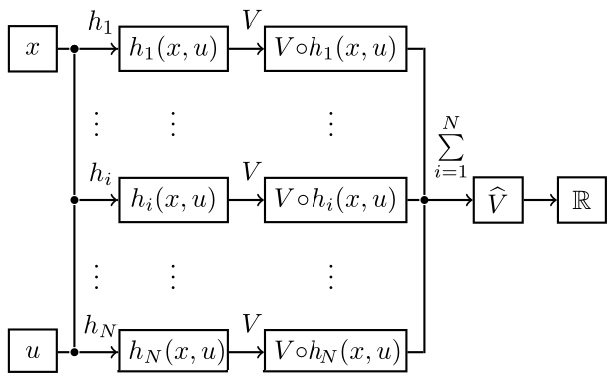}
		  \caption{Relationship between  $x,u,h_i,V,\widehat{V}$, and $\mathbb{R}$  for $\mathcal{I}=\{1,2,...,N\}$. }\label{Fig1}
\end{figure}

\begin{remark}\label{Remark2}  GISS-LFs have the following properties:
 \begin{enumerate}[(i)]

 \item Positive semidefiniteness. For a fixed $u$, the functional  $ \widehat{V}(x,u) $
  may be  degenerate at $x\neq 0$, namely,   $ \widehat{V} (x,u) =0$ at some $x\neq 0$. Thus, a GISS-LF is positive semidefinite.

  \item    Coercivity w.r.t. $x$.  By virtue of   condition~\eqref{condition-1},  the functional $\widehat{V}(x,u)$ is coercive  w.r.t. $x$  in the sense that for any fixed $u$,
      \begin{align*}
     \|x\|_{X}\to +\infty ~~ \Rightarrow~~    \widehat{V}(x,u)\to +\infty.
      \end{align*}
  \item    Non-coercivity w.r.t. $(x,u)$.
Note that $\widehat{V}(x,u)$ is non-coercive w.r.t. $(x,u)$ in the sense that
      \begin{align*}
     \|x\|_{X}+ \|u\|_{\mathcal{U}}\to +\infty  ~~\not\Rightarrow~~  \widehat{V}(x,u)\to +\infty.
      \end{align*}
  \end{enumerate}
\end{remark}

In the sequel, we always assume that conditions \textbf{(H1)} and \textbf{(H2)} are satisfied.

  A  generalized ISS Lyapunov theorem was proved in \cite{Zheng:2025}.

 \begin{theorem}[Generalized ISS Lyapunov theorem, \cite{Zheng:2025}]\label{main result}  If  the control system~$\Sigma$ admits a GISS-LF, then it is ISS.
\end{theorem}

The following result, which is a direct consequence of Theorem \ref{main result}, is   suited for ISS analysis of specific PDEs involving only first order time derivatives, in particular,  parabolic equations and first order hyperbolic equations.

\begin{corollary}[\cite{Zheng:2025}]\label{Corollary2}   {Suppose that     there exist} a continuous functionalal $V:X\to\mathbb{R}_{\geq 0}$,
 functions $\psi_1,\psi_2\in\mathcal{K}_{\infty}, \mu,\rho\in \mathcal{K}$, and positive constant $\lambda_0>0$   
  such that
 \begin{enumerate}
\item[(i)]  for  any  $x\in X$, it holds that
  \begin{align*}
0\leq &V(x)\leq \psi_1(\| x\|_X);
\end{align*}
\item[(ii)]   for  any   $x\in X$ and {any $z\in X$, it holds that}      
 \begin{align*}
  \psi_2\left(\|   x \|_X\right) \leq  { V(x-z )+V(-x-z )}+ \mu\left( \|z \|_{X}  \right);
 \end{align*}

\item[(iii)]  for  any    $u\in \mathcal{U}$, there is  {$z_u\in X$} satisfying {$ \|z_u\|_{X}  \leq \rho(\|u\|_{\mathcal{U}})$}, such  that the  derivative of $  V(x- {z_u})+V(-x-{z_u}) $  along the   trajectory satisfies
\begin{align}
 \dot{V}(x- {z_u})+\dot{V}(-x-{z_u})
 \leq  &-\lambda_0( V(x- {z_u})+V(-x-{z_u})).\label{Equ.10}
\end{align}

 \end{enumerate}
  Then, $V(x- {z_u})+V(-x-{z_u})$ is a GISS-LF and the control system~$\Sigma$ is ISS.
\end{corollary}

Furthermore, noting that $\mathcal{U}:=PC(\mathbb{R}_{\geq 0};U)$ (see condition \textbf{(H1)}),  we can prove the following result, which is more convenient for conducting  ISS analysis  of  PDEs.
\begin{proposition}\label{Prop. 2.2}
Under the conditions    of Corollary \ref{Corollary2} with (iii) replaced by
\begin{enumerate}
 \item[(iii')]  for  any $T\in\mathbb{R}_{>0}$ and   $u\in \mathcal{U}$, there is  {$z_{u,T}\in X$} satisfying  $$ \|z_{u,T}\|_{X}  \leq \rho\left(\sup_{t\in(0,T)}\|u(t)\|_{ {U}}\right),$$   such  that the  derivative of $  V(x- {z_{u,T}})+V(-x-{z_{u,T}}) $  along the   trajectory satisfies for all $t\in(0,T)$:
\begin{align}
  \dot{V}(x- {z_{u,T}})+\dot{V}(-x-{z_{u,T}})
 \leq  &-\lambda_0( V(x- {z_{u,T}})+V(-x-{z_{u,T}})),\label{Equ.10'}
\end{align}
\end{enumerate}
 the control system~$\Sigma$ is ISS,
   having the estimate for all $t\in \mathbb{R}_{>0}$:
   \begin{align}
  \psi_2\left(\|   \phi(t,x,u) \|_X\right)\leq &2 e^{-\lambda_0t}\psi_1(2\| x\|_X) +  \gamma \left( \sup_{s\in(0,t)}\|u(s)\|_{{U} }\right)  \label{Equ.11}
\end{align}
with    $   \gamma(\cdot):=  2\psi_1(2\rho(\cdot))+\mu(\rho(\cdot))\in \mathcal{K}_{\infty}$.
\end{proposition}

\begin{proof}
Let $\widehat{V}(x):= V(x- {z_{u,T}})+V(-x-{z_{u,T}})$.
%
  Applying the comparison principle (\cite[Proposition A.35, p.~328]{Mirochenko:2023book}) to \eqref{Equ.10'} and using condition (i),  we deduce that
\begin{align*}
  \widehat{V}(\phi(T,x,u))
 \leq &  e^{-\lambda_0 T} \widehat{V}(x)\notag\\
 \leq & e^{-\lambda_0 T}\left(\psi_1(\|x-z_{u,T}\|_X)+\psi_1(\|-x-z_{u,T}\|_X)\right) \\
  \leq & 2 e^{-\lambda_0 T}\left(\psi_1(\|x \|_X+ \| z_{u,T}\|_X)\right)\\
  \leq &2 e^{-\lambda_0 T}\left(\psi_1(2\|x \|_X) + \psi_1(2\| z_{u,T}\|_X)\right)\notag\\
  \leq & 2 e^{-\lambda_0 T }\psi_1(2\|x \|_X) +2 \psi_1(2\| z_{u,T}\|_X).
\end{align*}
Using   condition  (ii) and the assumption that $ \|z_{u,T}\|_{X}  \leq \rho\left(\sup_{t\in(0,T)}\|u(t)\|_{ {U}}\right)$, we obtain
\begin{align*}
 \psi_2\left(\|  \phi(T,x,u) \|_X\right)
 \leq&  \widehat{V}(\phi(T,x,u))  + \varphi\left( \|z_{u,T} \|_{X}  \right)   \\
   \leq &2 e^{-\lambda_0 T} \psi_1(2\|x \|_X) + 2\psi_1(2\| z_{u,T}\|_X)  + \mu\left( \|z_{u,T} \|_{X}  \right)\notag\\
   \leq & 2 e^{-\lambda_0 T} \psi_1(2\|x \|_X) + 2 \psi_1\left(2\rho\left(\sup_{t\in(0,T)}\|u(t)\|_{ {U}}\right)\right)   + \mu\left(\rho\left(\sup_{t\in(0,T)}\|u(t)\|_{ {U}}\right)\right),
\end{align*}
which, along with the arbitrariness of $T$, gives the desired result.
\end{proof}

\begin{remark}\label{Remark5}
By virtue   Proposition~\ref{Prop. 2.2}, we   call $  \widehat{V}(x):=V(x- {z_{u,T}})+V(-x-{z_{u,T}}) $ a GISS-LF for the control system~$\Sigma$ over the arbitrary time interval $(0,T)$. Moreover, when $T$ is not emphasized, we   simply call   $  \widehat{V}(x)$   a GISS-LF for the control system~$\Sigma$.
\end{remark}

It is worth noting that, due to the diversity and complexity of PDEs, constructing an explicit GISS-LF within a unified framework is extremely challenging,  and such constructions often require a case-by-case treatment.
Nevertheless, for the most representative classes of PDEs, in particular,  parabolic equations, first order hyperbolic equations, and wave equations, GISS-LFs can be constructed in a relatively unified manner. For these PDEs, a key technique lies in the use of Stampacchia's truncation to construct truncation functions and disturbance-dependent energy functionals $V(h_i(x,u))$. The number of such energy functionals is determined by the number of state variables  when the PDE is reformulated as an abstract system. Details of constructing GISS-LFs for these three classes of PDEs are provided in Sections \ref{GLF-parabolic},  \ref{GLF-hyperbolic}, and \ref{GLF-wave},  respectively.

%


\section{ISS analysis for $N$-D    parabolic equations via GLM}\label{GLF-parabolic}

In this section, we construct a GISS-LF for higher dimensional parabolic PDEs with different disturbances under nonlinear boundary conditions.

 In the sequel, let $N\geq 1$, and let $\Omega$ be an open bounded domain in $\mathbb{R}^N$ with smooth (when $N>1$) boundary   denoted by $\partial \Omega$. At each point $y\in\partial \Omega$, the unit outward normal vector is denoted by $\bm{\nu}(y)$.
 For a  differentiable {function  $h$}  defined over $\Omega$, $\nabla h$ denotes the gradient of $h$. 
 The notation $\diff S$ denotes the area element of $\partial \Omega$. The notation $ \mathscr{L}^N(\Omega) $ (or $ \mathscr{L}^{N-1}(\Gamma) $) denotes the $N$- (or  $(N-1)$-) dimensional Lebesgue  measure of $\Omega$ (or a set of $\Gamma\in \partial \Omega$).
 The notation $\Delta$ represents the standard Laplace operator.  For $T>0$, let $Q_T:=\Omega\times(0,T)$ and
 $Q_{\infty}:=\Omega\times\mathbb{R}_{>0}$.

  We consider the following higher dimensional nonlinear parabolic equation with both distributed in-domain and mixed  boundary disturbances:
  \begin{subequations}\label{parabolicPDE-d}
\begin{align}
w_t(y,t) =&\divi\   (a (y,t)\nabla w(y,t) )-c(y,t) \phi(w (y,t)) -{h(y,t,w(y,t) )} \notag\\
& +f(y,t),(y,t)\in Q_{\infty},\\
w(y,t) =&d_1(y,t),(y,t)\in \Gamma_1\times \mathbb{R}_{>0},\\
 a(y,t) \frac{\partial w(y,t)}{\partial \bm{\nu}} =&-\varphi(w(y,t))  +d_2(y,t),(y,t)\in \Gamma_2\times  \mathbb{R}_{>0},\\
w(y,0)=&w_0(y),y\in\Omega,
\end{align}
\end{subequations}
where
\begin{enumerate}
\item[$\bullet$]   $\Gamma_1$ and $\Gamma_2$ are  smooth  and disjoint  boundary sets ($(N-1)$-dimensional hypersurfaces) satisfying (see, e.g., Fig.~\ref{domain})
\begin{align*}
\Gamma_1\cup\Gamma_2=& \partial \Omega,\\
0\leq \mathscr{L}^{N-1}(\Gamma_1)\leq &\mathscr{L}^{N-1}(\partial \Omega);
 \end{align*}
 \item[$\bullet$] $a,c :\Omega\times\mathbb{R}_{> 0}\to {\mathbb{R}_{>  0}} $ are smooth functions satisfying
\begin{align*}
a(y,t ) \geq a_0 ,\forall (y,t )\in\Omega\times\mathbb{R}_{> 0},\\
c(y,t ) \geq c_0 ,\forall (y,t )\in\Omega\times\mathbb{R}_{> 0},
 \end{align*}
 with   positive constants $a_0 $ and $c_0$;
 \item[$\bullet$]   $\phi:  \mathbb{R}  \to\mathbb{R}$ is a  smooth  and strictly increasing function    satisfying
    \begin{subequations}\label{condition-phi}
    \begin{align}
   \phi(v)v\geq & 0, \forall v\in \mathbb{R} , \\
\phi(-v) \leq &-\phi(v),\forall v\in \mathbb{R}_{\geq 0} ,\label{condition-phi-2}\\
   \phi'(v)\geq& 1,\forall v\in \mathbb{R}  ;\label{condition-phi}
\end{align}
 \end{subequations}
  \item[$\bullet$] $h:\Omega\times\mathbb{R}_{> 0}\times\mathbb{R}\to\mathbb{R}$ is a  smooth  function satisfying
\begin{align}
h(y,t,v)v\geq 0,\forall (y,t,v)\in\Omega\times\mathbb{R}_{> 0}\times\mathbb{R} ;\label{condition-h}
\end{align}
\item[$\bullet$]   $ \varphi: \mathbb{R}  \to\mathbb{R}$ is a  smooth and strictly increasing function  satisfying
   \begin{subequations}\label{condition-varphi-total}
    \begin{align}
\varphi(v)v\geq &0,\forall v\in \mathbb{R}  ,\label{condition-varphi}\\
\varphi(-v) \leq &-\varphi(v),\forall v\in \mathbb{R}_{\geq 0}  ;\label{condition-varphi-2}
\end{align}
 \end{subequations}
 \item[$\bullet$]   $f,d_1,d_2:\Omega\times\mathbb{R}_{\geq  0}\to  \mathbb{R}$ are smooth functions, representing distributed in-domain and boundary disturbances, respectively;
       \item[$\bullet$] $w_0:\Omega \to  \mathbb{R}$ is a smooth function, representing the initial data.
\end{enumerate}

\begin{figure}
\centering

		\includegraphics[scale=0.4]{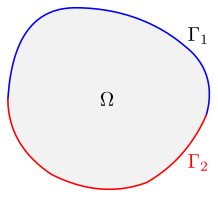} 	\ \ \ \ \ \ \ \ \ \ \ \ \ \ \ \ \ \ \ \ \ \ \ \ \ \ \ \ \ \ \includegraphics[scale=0.4]{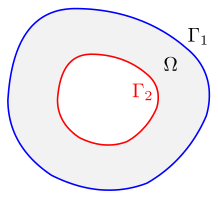}\\

\small{(a) Simply connected domain	 \ \ \ \ \ \ \ \ \ \ \ \ \ \ \ \ \ \ \ \ (b) Multiply connected domain}

\caption{Illustration of a domain     in $   \mathbb{R}^2$.}
\label{domain}
\end{figure}

We apply the technique of Stampacchia's truncation to construct a GISS-LF.
 For any $p\in(1,\infty)$,  let
\begin{subequations}\label{Def:G}
\begin{align}
g(s):=&  \left\{\begin{aligned}
& s^p,\ s\in \mathbb{R}_{\geq 0},\\
& 0,\ s\in \mathbb{R}_{< 0},
\end{aligned}\right.\label{18a}\\
G(s):=&\int_0^sg(\tau)\diff \tau
= \left\{\begin{aligned}
& \frac{1}{p+1}s^{p+1},\ s\in \mathbb{R}_{\geq 0},\\
& 0,\ s\in \mathbb{R}_{< 0}.\end{aligned}\right.
\end{align}
\end{subequations}

Note that  $g$ and $G$ have the following properties, which will be extensively used in this paper:
\begin{enumerate}
\item[\textbf{(G1)}]   $G(s)=g(s)= 0  $  for all $s\in \mathbb{R}_{\leq 0}$;
\item[\textbf{(G2)}]  $G(s)=\frac{1}{p+1}g(s)s$ for all $s\in \mathbb{R} $;
\item[\textbf{(G3)}]    $g'(s)\geq 0$  for all $s\in \mathbb{R} $;
\item[\textbf{(G4)}]$G(s )\leq    G(s+\tau)+G(s-\tau) $ for all $s,\tau\in \mathbb{R} $;
\item[\textbf{(G5)}]$G(|s|+\tau )\leq    G(s+\tau)+G(-s+\tau) $ for all $s,\tau\in \mathbb{R} $;
\item[\textbf{(G6)}]$G(s+\tau)\leq  2^p(G(s)+G(\tau))$ for all $s,\tau\in \mathbb{R} $;
\item[\textbf{(G7)}]  $G(|s|)\leq   2^p (G(s+\tau-M)+G(s-\tau-M) +G(-s+\tau-M)+G(-s-\tau-M)) 
   +2^pG(M) $ for all $ s,\tau,M\in\mathbb{R}$;
 
\item[\textbf{(G8)}] Young's inequality:
\begin{align*}
 g(s)\tau\leq \varepsilon  G(s)+\left(\frac{p}{\varepsilon}\right) ^{ p}G(\tau),\forall s,\tau\in \mathbb{R} , \varepsilon\in \mathbb{R}_{>0}.
 \end{align*}
\end{enumerate}
The proofs of properties \textbf{(G1)}-\textbf{(G8)} are provided in Appendix \ref{Appendix: G}.

 Since $\phi$ and $\varphi$ are strictly increasing,   their inverses, denoted by $\phi^{-1}$ and $\varphi^{-1}$, respectively, exist.   Moreover,  $\phi^{-1},\varphi^{-1}\in\mathcal{K}$. For any $T>0$, let
\begin{align*}
 M:=&\phi^{-1}\left(\frac{1}{c_0} \|f\|_{\infty,T}\right)
  + \|d_1\|_{\infty,T}
 +\varphi^{-1}\left(\|d_2\|_{\infty,T}\right),
\end{align*}
where, for notational simplicity, we denote
\begin{align*}
 \|f\|_{\infty,T}:=&\sup_{(y,t)\in Q_T}|f(y,t)|  ,\\
\|d_1\|_{\infty,T}:=& \sup_{(y,t)\in \Gamma_1\times(0,T)}|d_1(y,t)| ,\\
 \|d_2\|_{\infty,T}:=& \sup_{(y,t)\in \Gamma_2\times(0,T)}|d_2(y,t)|.
\end{align*}
It is clear that
\begin{align*}
 \|f\|_{\infty,T} \leq& c_0\phi\left(M\right), \\
 \|d_1\|_{\infty,T}\leq& M,\\
\|d_2\|_{\infty,T}\leq& \varphi(M).
\end{align*}

Let
\begin{align*}
V(w):=& \int_{\Omega}G(w(y,t))\diff y,
\end{align*}
  and
  \begin{align}
\widehat{V}(w):= V (w -M) +V (-w -M).\label{GISS-LF-V}
\end{align}
Note that $\widehat{V}(w)= 0$ for  $|w|\leq M$. Thus, $\widehat{V}(w)$ is     positive semidefinite.

 For $\widehat{V}(w)$ and system~\eqref{parabolicPDE-d}, we have the following result.
 \begin{proposition}\label{Prop.PDE1}
 The following statements hold true:

 \begin{itemize}
\item[(i)]
 For any $p\in (1,\infty)$, the functional $ \widehat{V} (w)$ defined by \eqref{GISS-LF-V} is a GISS-LF for system~\eqref{parabolicPDE-d} over the arbitrary time interval $(0,T)$ and hence,   system~\eqref{parabolicPDE-d} is ISS in the   $L^{p+1}$-norm.
 \item[(ii)]
Furthermore,    system~\eqref{parabolicPDE-d} is ISS in the   $L^{q}$-norm for any $q\in [2,\infty]$,   having the estimate
   \begin{align}
 \|  w(\cdot,T) \|_{L^{q}(\Omega)}
     \leq& 4   \|  w_0 \|_{L^{q}(\Omega)}e^{-c_0 T}+8M,\forall T\in \mathbb{R}_{>0}.\label{q-estimate}
\end{align}
 \end{itemize}
\end{proposition}
\begin{proof} We employ Proposition~\ref{Prop. 2.2} to prove this proposition.

First, it is clear that the conditions (i) and (ii) of Corollary~\ref{Corollary2} are fulfilled with (see \cite[Proposition 2]{Zheng:2025})
   \begin{align*}
   \psi_1(s):=&\frac{1}{p+1}  s^{p+1} ,\forall s\in \mathbb{R}_{\geq 0},\\
     \psi_2(s):=&\frac{1}{ 2^{p}(p+1)}s^{p+1},\forall s\in \mathbb{R}_{\geq 0},\\
     \mu(s):=&2s^{p+1},\forall s\in \mathbb{R}_{\geq 0}.
  \end{align*}
We shall verify the condition (iii') of Proposition~\ref{Prop. 2.2}.
Indeed,
for any $t\in(0,T)$, using the Green's formula (see \cite[Theorem~3~(ii), p.712]{Evans:2010}), we deduce that
\begin{align}
 \dot{V}(w-M)
 =& \frac{\diff~}{\diff t} \int_{\Omega} G(w-M) \diff y \notag\\
= &  \int_{\Omega} g(w-M)w_t\diff y \notag\\
= &  \int_{\Omega} g(w-M)(\divi\   (a w )-c \phi(w)
 -{h(y,t,w )}+f  )\diff y\notag\\
= &   \int_{\Gamma_1} ag(w-M)\frac{\partial w}{\partial \bm{\nu}}\diff S+ \int_{\Gamma_2} ag(w-M)\frac{\partial w}{\partial \bm{\nu}}\diff S - \int_{\Omega}a g'(w-M)|\nabla  w|^2  \diff y  \notag\\
&- \int_{\Omega} g(w-M) (c  \phi(w)-f)  \diff y  - \int_{\Omega} g(w-M) h(y,t,w)  \diff y.\label{parabolic-14}
 \end{align}

 Using     the boundary conditions, the definition   of $M$, \eqref{18a}, and the monotonicity of $\varphi$,  we have
 \begin{align*}
 g(w-M)=&0 \  \text{on}\  \Gamma_1\times (0,T)
 \end{align*}
 and
 \begin{align*}
   g(w-M)\frac{\partial w}{\partial \bm{\nu}}
  =&-g(w-M)(\varphi (w) -d_2)\notag\\
    \leq& -g(w-M)\left( \varphi (w) -\|d_2\|_{\infty,T}\right)\notag\\
     \leq& -g(w-M)\left( \varphi (w) -\varphi (M)\right)\notag\\
    \leq&     0 \  \text{on}\   \Gamma_2\times (0,T).
 \end{align*}
The condition \eqref{condition-h} ensures that
\begin{align*}
 h(y,t,v)  \geq &0,\forall (y,t,v)\in\Omega\times(0,T)\times\mathbb{R}_{\geq 0} ,
 \end{align*}
which, along with the  definition  $g$, ensures that
\begin{align*}
-  g(w-M) h(y,t,w)  \leq &0,\forall (y,t,w)\in\Omega\times(0,T)\times\mathbb{R} .
 \end{align*}
In addition, we deduce that
\begin{align*}
 -\int_{\Omega} g(w-M)(c\phi(w)- f)\diff y
 \leq&  -\int_{\Omega} g(w-M)\left(c\phi(w)-  \|f\|_{\infty,T}\right)\diff y\notag\\
  \leq&  -\int_{\Omega} g(w-M)\left(c\phi(w)-  c_0\phi\left(M\right)\right)\diff y\notag\\
   \leq&  -c_0\int_{\Omega} g(w-M)\left( \phi(w)-  \phi\left(M\right)\right)\diff y\notag\\
    =&  -c_0\int_{\Omega}  g(w-M)\phi'(v)\left( w-  M\right)\diff y\notag\\
 \leq & -c_0\int_{\Omega} g(w-M)(w-  M)\diff y\quad (\text{by}\ \eqref{condition-phi})\notag\\
 =&-c_0(p+1)\int_{\Omega} G(w-M) \diff y, \quad (\text{by}\ \textbf{(G2)})
\end{align*}
where $v$ is a function between $w$ and $M$.

Then, in view of \textbf{(G3)}, for any $t\in(0,T)$, we get
\begin{align}
 \dot{V}(w-M)
\leq &  -c_0(p+1)  \int_{\Omega} G(w-M)     \diff y\notag \\
=& -c_0(p+1) V(w-M).\label{parabolic-16}
\end{align}

Now let $\widetilde{w}:=-w$, which satisfies
  \begin{subequations}\label{parabolicPDE-d'}
\begin{align}
\widetilde{w}_t(y,t) =&\divi\   (a (y,t)\nabla \widetilde{w}(y,t) ) +c(y,t) \phi(-\widetilde{w} (y,t)) +{h(y,t,-\widetilde{w}(y,t) )}\notag\\
&  -f(y,t),(y,t)\in Q_{\infty},\\
\widetilde{w}(y,t) =&-d_1(y,t),(y,t)\in \Gamma_1\times \mathbb{R}_{>0},\\
 a(y,t) \frac{\partial \widetilde{w}(y,t)}{\partial \bm{\nu}} =&\varphi(-\widetilde{w}(y,t))  -d_2(y,t),(y,t)\in \Gamma_2\times  \mathbb{R}_{>0},\\
\widetilde{w}(y,0)=&-w_0(y),y\in\Omega,
\end{align}
\end{subequations}

 Note that
 \begin{align*}
 g(\widetilde{w}-M)=&0 \  \text{on}\  \Gamma_1\times (0,T)
 \end{align*}
 and
\begin{align*}
    g(\widetilde{w}-M)\frac{\partial \widetilde{w}}{\partial \bm{\nu}}
   =& g(\widetilde{w}-M)(\varphi (-\widetilde{w}) -d_2)\notag\\
    \leq&  g(\widetilde{w}-M)\left( \varphi (-\widetilde{w}) +\|d_2\|_{\infty,T}\right)\notag\\
     \leq&  g(\widetilde{w}-M)\left( \varphi (-\widetilde{w}) +\varphi (M)\right)\  \text{on}\   \Gamma_2\times (0,T).
 \end{align*}
Furthermore, in view of the nonnegativity of $g$ and the degeneracy property \textbf{(G1)},   the condition of $\varphi $ (see \eqref{condition-varphi-2}) and the increasing monotonicity of $\varphi$ imply that
 \begin{align*}
    g(\widetilde{w}-M)\frac{\partial \widetilde{w}}{\partial \bm{\nu}}
       \leq&  g(\widetilde{w}-M)\left( \varphi (-\widetilde{w}) +\varphi (M)\right)
       \notag\\
       \leq &  g(\widetilde{w}-M)\left( -\varphi (\widetilde{w}) +\varphi (M)\right)\\
    \leq&     0 \  \text{on}\   \Gamma_2\times (0,T).
 \end{align*}

Similarly, it holds that
 \begin{align*}
   \int_{\Omega} g(\widetilde{w}-M)(c\phi(-\widetilde{w})+ f)\diff y
  \leq&   \int_{\Omega} g(\widetilde{w}-M)\left(-c\phi(\widetilde{w})+  c_0\phi\left(M\right)\right)\diff y\notag\\
   \leq&  -c_0\int_{\Omega} g(\widetilde{w}-M)\left( \phi(\widetilde{w})-  \phi\left(M\right)\right)\diff y\notag\\
     =&  -c_0\int_{\Omega}  g(\widetilde{w}-M)\phi'(v)\left( \widetilde{w}-  M\right)\diff y\notag\\
 \leq &-c_0(p+1)\int_{\Omega} G(\widetilde{w}-M) \diff y,\forall t\in (0,T),
\end{align*}
where $v$ is a function between $\widetilde{w}$ and $M$.

In addition, due to the fact that $
 h(y,t,v)  \geq  0 $   for all $(y,t,v)\in\Omega\times(0,T)\times\mathbb{R}_{\leq 0}
$, we have
\begin{align*}
  g(\widetilde{w}-M) h(y,t,-\widetilde{w})  \leq &0,\forall (y,t,w)\in\Omega\times(0,T)\times\mathbb{R} .
 \end{align*}
 Analogous to \eqref{parabolic-16}, for any $t\in(0,T)$, we have
 \begin{align}
 \dot{V}(\widetilde{w}-M)
\leq &  -c_0(p+1)  \int_{\Omega} G(\widetilde{w}-M)     \diff y \notag\\
=& -c_0(p+1) V(\widetilde{w}-M), \label{parabolic-17}
\end{align}
 which, along with \eqref{parabolic-16}, indicates that
 the condition (iii') of Proposition~\ref{Prop. 2.2} is fulfilled   with $x:=w$, $u:=(f,d_0,d_1)$, $z_{u,T}:=M$, and $\rho\left(\sup_{t\in(0,T)}\|u(t)\|_{ {U}}\right):=M$.

Applying Proposition~\ref{Prop. 2.2} and Remark~\ref{Remark5}, we deduce that the functional $ \widehat{V} (w)$ defined by \eqref{GISS-LF-V} is a GISS-LF for system~\eqref{parabolicPDE-d} and hence,  system~\eqref{parabolicPDE-d} is ISS in the   $L^{p+1}$-norm with arbitrary $p\in (1,\infty)$.

 Furthermore, we deduce by \eqref{Equ.11}   that
 \begin{align*}
  \psi_2\left( \|  w(\cdot,T) \|_{L^{p+1}(\Omega)}\right)\leq &2 e^{-c_0(p+1)T}\psi_1\left(2\|  w_0 \|_{L^{p+1}(\Omega)}\right) +   2\psi_1(2M)+\mu(M),
\end{align*}
which implies that
  \begin{align*}
  \|  w(\cdot,T) \|_{L^{p+1}(\Omega)}
   \leq & \Big( 4  \|  w_0 \|_{L^{p+1}(\Omega)} e^{-c_0 T}  + 2M+ \left(  p+1  \right)^{\frac{1}{p+1}}2M\Big)\notag\\
  \leq& 4   \|  w_0 \|_{L^{p+1}(\Omega)} e^{-c_0 T}+8M.
\end{align*}
Thus,  system~\eqref{parabolicPDE-d} is ISS in the   $L^{p+1}$-norm.

Furthermore,  due to the arbitrariness of $p\in(1,\infty)$, we conclude that  system~\eqref{parabolicPDE-d} is ISS in the   $L^{q}$-norm for any $q\in[2,\infty]$, having the estimate \eqref{q-estimate}.
\end{proof}
\begin{remark} Usually  there is no need to verify  each condition of Proposition~\ref{Prop. 2.2} (or conditions~(i)-(iii) in Definition~\ref{Def:GISS-LF}) so as to obtain the ISS  of a PDE. Indeed, one may directly compute $\dot{\widehat{V}} (x,u)$   and using the Gronwall's inequality  to establish an explicit  ISS  estimate, where $x$ and $u$ denote   the state and input of  the considered PDE   under an abstract form.
\end{remark}

 \section{ISS analysis     for {$1$-D first order hyperbolic}  equations via GLM} \label{GLF-hyperbolic}
 In this section, we employ the GLM to establish  ISS estimates
for the following  $1$-D hyperbolic equation:
 \begin{subequations}\label{first hyperbolic PDE}
 \begin{align}
 \rho_t(y,t)+\left(\lambda(W(t))\rho(y,t)\right)_y=&0,  (y,t)\in (0,1)\times\mathbb{R}_{>0},\\
 \rho(0,t)-k\rho(1,t)=&d(t),t\in \mathbb{R}_{>0},\\
 \rho(y,0)=&\rho_0(y),y\in(0,1),
 \end{align}
 \end{subequations}
where  $\lambda\in C^1(\mathbb{R};\mathbb{R}_{>0})$, $k\in(-1,1)$ is a constant,  $d,\rho_0$ are smooth functions, and $W(t):=\int_0^1\rho(y,t)\diff y$.

 Equation~\eqref{first hyperbolic PDE}   is commonly used to model semiconductor manufacturing systems; see, e.g., \cite{Armbruster:2006,Herty:2007,Marca:2010}.  The key feature of this PDE is characterized by the velocity function $\lambda$ depending on the total mass $W(t)$.

We propose either one of the following assumptions on  $\lambda$:
\begin{enumerate}
\item[\textbf{(A1)}]   $\lambda(s)\geq \lambda_0$ for all $s\in \mathbb{R}$, where  $\lambda_0$ is  a positive constant;
    \item[\textbf{(A2)}]   $\lambda(s)\geq \lambda(|s|)$ for all $s\in \mathbb{R}$ and $\lambda(s)$ is decreasing in $s\in \mathbb{R}_{\geq 0}$.
     \end{enumerate}
     Typical forms of $\lambda(s)$ satisfying  \textbf{(A2)} include
      \begin{align*}
     \lambda(s):=Ae^{-s^2}+B,\forall s\in  \mathbb{R}_{\geq 0} ,
    \end{align*}
     and
    \begin{align*}
     \lambda(s):=\frac{A}{B+s},\forall s\in  \mathbb{R}_{\geq 0} ,
    \end{align*}
     where $A$ and $B$ are positive constants.

For any $p\in(1,\infty)$, let $g$ and $G$ be defined by \eqref{Def:G}.
For any $T>0$, let
\begin{align*}
  M:=   \frac{1}{1-|k|}\sup_{t\in  (0,T)}|d (t)|
  =
     \left\{\begin{aligned}
& \frac{1}{1-k}\sup_{t\in  (0,T)}|d (t)|,\ k\in[0,1),\\
& \frac{1}{1+k}\sup_{t\in  (0,T)}|d (t)|,\ k\in(-1,0).
\end{aligned}\right.
\end{align*}

Let
\begin{align*}
V(\rho):= \int_0^1e^{-r  y}G(\rho(y,t))\diff y,
\end{align*}
where $r$ is a positive constant satisfying
  \begin{align}
 0< r\leq  (p+1)\ln \frac{1}{|k|}.\label{condition-sigma}
\end{align}
 Let
  \begin{align}
\widehat{V}(\rho):= V (\rho -M) +V (-\rho -M).\label{GISS-LF-V-2}
\end{align}

The first result for system~\eqref{first hyperbolic PDE} is  regarded as a global ISS property.

 \begin{proposition}\label{Prop. PDE2-1} Under the assumption \textbf{(A1)}, the following statements hold true.
 \begin{enumerate}[(i)]
\item
 For  any $p\in (1,\infty)$ and $r \in \mathbb{R}_{>0}$ satisfying \eqref{condition-sigma}, the functional $ \widehat{V} (\rho)$ defined by \eqref{GISS-LF-V-2} is a GISS-LF of system~\eqref{first hyperbolic PDE} and thus,   system~\eqref{first hyperbolic PDE} is ISS in the   $L^{p+1}$-norm. In particular, system~\eqref{first hyperbolic PDE}  admits the following   estimate:
  \begin{align}
  \|\rho(\cdot,t)\|_{L^{p+1}(0,1)}
  \leq &   2 e^{\frac{r}{p+1}}\|\rho_0 \|_{L^{p+1}(0,1)} e^{-\frac{r  \lambda_0}{p+1} t}
 + 2  \sup_{s\in(0,t)} |d (s)|     ,\forall t\in \mathbb{R}_{>0}.\label{PDE2-ISS-1}
\end{align}
\item  Furthermore,  for any $q\in[2,\infty]$, system~\eqref{first hyperbolic PDE} is ISS in the   $L^{q}$-norm and admits  the following 
    estimate:
 \begin{align}
  \|\rho(\cdot,t)\|_{L^{q}(0,1)}
 \leq &   \frac{ 2  }{|k|}   \|\rho_0 \|_{L^{q}(0,1)} |k|^{\frac{\lambda_0  t}{2}}
 +\frac{2}{1-|k|}\sup_{s\in  (0,t)}|d (s)| ,\forall t\in \mathbb{R}_{>0}.  \label{PDE2-ISS-2}
 \end{align}
 \end{enumerate}
 \end{proposition}
\begin{remark}
Note that $|k|^{\frac{\lambda_0  t}{2}}=e^{-k_0t}$ with $k_0:=\frac{\lambda_0}{2}\ln \frac{1}{|k|}>0$. Thus, similarly to \eqref{PDE2-ISS-1},  estimate \eqref{PDE2-ISS-2} also yields the exponential ISS  for   system~\eqref{first hyperbolic PDE}.
\end{remark}
Using the GLM, we may establish  local ISS estimates   for system~\eqref{first hyperbolic PDE} with a general $\lambda$.
 \begin{proposition}\label{Prop. PDE2-2}   Under the assumption \textbf{(A2)}, the following statements hold true.
 \begin{enumerate}[(i)]
\item
 For  any $p\in (1,\infty)$ and $r \in \mathbb{R}_{>0}$ satisfying \eqref{condition-sigma},  system~\eqref{first hyperbolic PDE} is LISS in the   $L^{p+1}$-norm, having the following   estimate:
  \begin{align*}
 \|\rho(\cdot,t)\|_{L^{p+1}(0,1)}
  \leq &   2 e^{\frac{r}{p+1}}\|\rho_0 \|_{L^{p+1}(0,1)} e^{-\frac{r  \Lambda_0}{p+1} t} + 2  \sup_{s\in  (0,t)}|d (s)|   ,\forall t\in \mathbb{R}_{>0},
\end{align*}
provided that
\begin{align}
  \|\rho_0 \|_{L^{p+1}(0,1)}  +   \sup_{t\in  \mathbb{R}_{\geq 0}}|d (t)|\leq R_0\label{PDE2-31}
 \end{align}
 with some positive constant $R_0$,
where
\begin{align}
 \Lambda_0:= \lambda\left( 2R_0\max\left\{\frac{ 1   }{|k|}, \frac{ 1   }{1-|k|} \right\}\right)>0.\label{Lambda}
\end{align}
\item  Furthermore, for any $q\in[2,\infty]$, system~\eqref{first hyperbolic PDE} is LISS in the   $L^{q}$-norm, having the following  estimate:
 \begin{align*}
  \|\rho(\cdot,t)\|_{L^{q}(0,1)}
 \leq &   \frac{ 2  }{|k|}   \|\rho_0 \|_{L^{q}(0,1)} |k|^{\frac{\Lambda_0  t}{2}}   +\frac{2}{1-|k|}\sup_{s\in  (0,t)}|d (s)| ,\forall t\in \mathbb{R}_{>0},
 \end{align*}
 provided that
\begin{align*}
  \|\rho_0 \|_{L^{q}(0,1)}  +   \sup_{t\in  \mathbb{R}_{\geq 0}}|d (t)|\leq R_0
 \end{align*}
 with some positive constant $R_0$,
 where $\Lambda_0$ is the same as in \eqref{Lambda}.
  \end{enumerate}
 \end{proposition}

Now, we employ the GLM to prove Proposition~\ref{Prop. PDE2-1} and Proposition~\ref{Prop. PDE2-2}, respectively.

\begin{proof2} 
We proceed the proof with three steps.

\textbf{Step 1:} We show that the functional $ \widehat{V} (\rho)$ defined by \eqref{GISS-LF-V-2} is a GISS-LF for system~\eqref{first hyperbolic PDE}.

First,   analogous to \cite[Proposition 2.]{Zheng:2025},  it is direct to verify that the conditions (i) and (ii) of Corollary~\ref{Corollary2} are fulfilled. Next, for  any $t\in(0,T)$, we claim that
\begin{align}
  \dot{\widehat{V}}(\rho)
 \leq&
  -r \lambda(W)\widehat{V}(\rho) \label{36}\\
  \leq&  -r \lambda_0\widehat{V}(\rho)
   \forall t\in (0,T),
   \label{PDE2-hatV''}
\end{align}
which, along with  the conditions (i) and (ii) of Corollary~\ref{Corollary2}, ensures that  $ \widehat{V} (\rho)$ is a GISS-LF.

Now we prove \eqref{PDE2-hatV''}. Indeed, for any $t\in(0,T)$, using integrating by  parts, we obtain
\begin{align}
 \dot{V}(\rho-M)
=&\int_0^1e^{-r y}g(\rho-M)\rho_t\diff y\notag\\
= &-\int_0^1e^{-r y}g(\rho-M) \left(\lambda(W)\rho(y,t)\right)_y \diff y\notag\\
=&-\lambda(W)\int_0^1e^{-r y}\left(G(\rho-M) \right)_y \diff y\notag\\
=&-\lambda(W)\left[e^{-r y}G(\rho-M)\right]_{y=0}^{y=1} -r \lambda(W)\int_0^1e^{-r y} G(\rho-M) \diff y\notag\\
=&-\lambda(W)\left(e^{-r}G(\rho(1,t)-M)- G(\rho(0,t)-M)\right) -r \lambda(W)V(\rho-M)\notag\\
=&- \lambda(W) e^{-r}G(\rho(1,t)-M)    +\lambda(W)G(k\rho(1,t)-M+d(t)) -r \lambda(W)V(\rho-M).\label{PDE2-V1}
\end{align}

Analogously,   we have
\begin{align}
  \dot{V}(-\rho-M)
=&- \lambda(W) e^{-r}G(-\rho(1,t)-M)  +\lambda(W)G(-k\rho(1,t)-M-d(t)) \notag\\
 &-r \lambda(W)V(-\rho-M), 
 \forall t\in (0,T).\label{PDE2-V2}
\end{align}
Combining  \eqref{PDE2-V1} and \eqref{PDE2-V2}, we have
\begin{align}
 \dot{\widehat{V}}(\rho)
 \leq
 & \lambda(W) G(k\rho(1,t)-M+d(t)) +\lambda(W)G(-k\rho(1,t)-M-d(t))  - \lambda(W) e^{-r}  G(\rho(1,t)-M) \notag\\
  &- \lambda(W) e^{-r}  G(-\rho(1,t)-M)    -r \lambda(W)\widehat{V}(\rho),\forall t\in (0,T).\label{PDE2-hatV}
\end{align}

Now we   prove \eqref{PDE2-hatV''} in two cases: $k\in[0,1)$ and $k\in(-1,0)$.

  \emph{Case 1:  $k\in[0,1)$.}  Note that $G$ has the following property:
\begin{align}\label{G-m}
G(ms) =  m^{p+1}G(s),\ \forall m\in\mathbb{R}_{\geq 0},  s\in\mathbb{R}.
\end{align}
 Using the definition of $M$, the monotonicity of $G$, and    \eqref{G-m}, we have
\begin{align}
  G(k\rho(1,t)-M+d(t))
 \leq &  G(k\rho(1,t)-M+|d(t)|)\notag\\
 \leq &
 G(k\rho(1,t)-M+(1-k)M) \notag\\
= &G( k(\rho(1,t)-  M) )\notag\\
 =&k^{p+1}G (\rho(1,t)-  M),\forall t\in (0,T),\label{PDE2-34}
 \end{align}
 and
 \begin{align}
 G(-k\rho(1,t)-M-d(t))
 \leq &  G(-k\rho(1,t)-M+|d(t)|)\notag\\
\leq &
   G(-k\rho(1,t)-M+(1-k)M) \notag\\
= &     G(k(-\rho(1,t) - M)) \notag\\
= &k^{p+1}  G(-\rho(1,t) - M) ,\forall t\in (0,T).\label{PDE2-35}
\end{align}

Putting  \eqref{PDE2-34} and \eqref{PDE2-35} into \eqref{PDE2-hatV}, we get
\begin{align}
  \dot{\widehat{V}}(\rho)
 \leq
 & \lambda(W)  \left(k^{p+1}-e^{-r}\right)G( \rho(1,t)-M)   +\lambda(W)   \left(k^{p+1}-e^{-r}\right)G( -\rho(1,t)-M)  \notag\\
  & -r \lambda(W)\widehat{V}(\rho),\forall t\in (0,T),
   \forall t\in (0,T).\label{PDE2-hatV'}
\end{align}
 Note that the choice of $r$ (see \eqref{condition-sigma}) implies that
\begin{align*}
k^{p+1}- e^{-r}\leq 0 .
\end{align*}
  Then, \eqref{PDE2-hatV'} leads to \eqref{PDE2-hatV''}.

 \emph{Case 2:  $k\in(-1,0)$.} Analogous in Case 1, using the definition of $M$, the monotonicity of $G$, and   \eqref{G-m}, we have
\begin{align}
  G(k\rho(1,t)-M+d(t))
 \leq &
 G(k\rho(1,t)-M+(1-|k|)M) \notag\\
= &G( k(\rho(1,t)+  M) )\notag\\
 =&(-k)^{p+1}G (-\rho(1,t)- M),\forall t\in (0,T),\label{PDE2-34'}
 \end{align}
 and
 \begin{align}
  G(-k\rho(1,t)-M-d(t))
\leq &
   G(-k\rho(1,t)-M+(1-|k|)M) \notag\\
 = &     G(k(-\rho(1,t) + M)) \notag\\
= &(-k)^{p+1}  G( \rho(1,t) - M) ,\forall t\in (0,T).\label{PDE2-35'}
\end{align}

Putting  \eqref{PDE2-34'} and \eqref{PDE2-35'} into \eqref{PDE2-hatV}, we get
\begin{align}
  \dot{\widehat{V}}(\rho)
 \leq
 & \lambda(W)  \left((-k)^{p+1}-e^{-r}\right)G( \rho(1,t)-M)    +\lambda(W)   \left((-k)^{p+1}-e^{-r}\right)G( -\rho(1,t)-M)  \notag\\
  & -r \lambda(W)\widehat{V}(\rho),\forall t\in (0,T),
   \forall t\in (0,T).\label{PDE2-hatV'''}
\end{align}
 Note that the choice of $r$ (see \eqref{condition-sigma}) implies that
\begin{align*}
(-k)^{p+1}=|k|^{p+1}\leq e^{-r} .
\end{align*}
  Then, \eqref{PDE2-hatV'''} leads to \eqref{PDE2-hatV''}. We conclude that \eqref{PDE2-hatV''} holds true for all $k\in(-,1,1)$.

\textbf{Step 2:} We prove the ISS estimate \eqref{PDE2-ISS-1}.
  Indeed, applying  the Gronwall's inequality to \eqref{PDE2-hatV''}, we obtain
\begin{align}
 {\widehat{V}}( \rho(y,T) )
 \leq&
 e^{-r  \lambda_0 T}\widehat{V}(\rho_0(y)).\notag
\end{align}
Therefore,  using \textbf{(G6)} and \textbf{(G5)},  we have
\begin{align}
 \frac{1}{p+1}\int_0^1|\rho(y,T)|^{p+1}\diff y
 = &  \int_0^1G\left(|\rho(y,T)|\right)\diff y\notag\\
\leq &2^{p}\int_0^1G\left(|\rho(y,T)|-M  \right)\diff y+2^{p}G(M)\notag\\
= &2^{p}\int_0^1\left(G\left( \rho(y,T)-M  \right)+G\left(- \rho(y,T)-M  \right)\right)\diff y  +2^{p}G(M)\notag\\
\leq  &2^{p}e^{r}\left(V\left( \rho(y,T)-M  \right)+ V\left( -\rho(y,T)-M  \right)\right)+2^{p}G(M)\notag\\
=&2^{p}e^{r}{\widehat{V}}( \rho(y,T) )+2^{p}G(M)\notag\\
\leq &2^{p}e^{r}e^{-r  \lambda_0 T}\widehat{V}(\rho_0(y))+2^{p}G(M)\notag\\
=&2^{p} e^{r} e^{-r  \lambda_0 T}\int_0^1e^{-r  y} \left(G(\rho_0(y)-M)+G(-\rho_0(y)-M) \right) \diff y +2^{p}G(M)\notag\\
= &2^{p}e^{r}e^{-r  \lambda_0 T} \int_0^1e^{-r  y} G(|\rho_0(y)|-M) \diff y +2^{p}G(M)\notag\\
\leq &2^{p}e^{r}e^{-r  \lambda_0 T} \int_0^1  G(|\rho_0(y)| ) \diff y +2^{p}G(M)\notag\\
=& \frac{2^{p}e^{r}}{p+1}e^{-r  \lambda_0 T} \int_0^1|\rho_0(y)|^{p+1}\diff y +\frac{2^{p}}{p+1}M^{p+1}, \label{PDE2-26}
\end{align}
which  implies that
\begin{align}
 \|\rho(\cdot,T)\|_{L^{p+1}(0,1)}
 \leq &  2^{\frac{p}{p+1}}e^{\frac{r}{p+1}} e^{-\frac{r  \lambda_0}{p+1} T}\|\rho_0 \|_{L^{p+1}(0,1)}  + 2^{\frac{p}{p+1}} M,\forall T\in \mathbb{R}_{>0}. \label{PDE2-27}
\end{align}
Therefore, the ISS estimate \eqref{PDE2-ISS-1} holds true.

\textbf{Step 3:} We prove the ISS estimate \eqref{PDE2-ISS-2}. Now let $r$  satisfy
\begin{align}
 0<-\frac{1}{2}(p+1)\ln |k|\leq r\leq -(p+1)\ln |k|,\label{condition-sigma'}
\end{align}
which ensures that
\begin{align*}
 e^{\frac{r}{p+1}} \leq   \frac{1}{|k|}   \ \ \text{and}\ \
  e^{-\frac{r}{p+1}} \leq    |k|^{\frac{1}{2}} .
\end{align*}
For any $T\in \mathbb{R}_{>0}$, it follows from \eqref{PDE2-27} that
\begin{align}
  \|\rho(\cdot,T)\|_{L^{p+1}(0,1)}
 \leq &   \frac{ 2   }{|k|}   |k|^{\frac{\lambda_0  T}{2}}  \|\rho_0 \|_{L^{p+1}(0,1)}  + 2  M .\label{PDE2-28}
\end{align}
Therefore,   the ISS estimate \eqref{PDE2-ISS-2} holds true for any $q:=p+1 \in (2,\infty)$.

letting $p\to 1$ and $p\to \infty$ in \eqref{PDE2-28} yields that the ISS estimate~\eqref{PDE2-ISS-2} holds for $q=2$ and $q=\infty$, respectively.
\end{proof2}

\begin{proof3}
 For any $T\in \mathbb{R}_{>0}$,    applying  the Gronwall's inequality to \eqref{36}, we obtain
\begin{align}
 {\widehat{V}}( \rho(y,T) ) \leq&  e^{-r \int_0^T\lambda(W(t))\diff t}\widehat{V}(\rho_0(y)).\label{hatV-W(t)}
\end{align}
 %
%
%
It follows that
\begin{align}
 {\widehat{V}}( \rho(y,T) ) \leq&   \widehat{V}(\rho_0(y)),
 \end{align}
which, along with similar derivations for \eqref{PDE2-26}, \eqref{PDE2-27}, and \eqref{PDE2-28}  (by setting $\lambda_0=0$ in \eqref{PDE2-26}, \eqref{PDE2-27}, and \eqref{PDE2-28}),  leads
 to
 \begin{align*}
 \|\rho(\cdot,T)\|_{L^{p+1}(0,1)}
 \leq &  \frac{ 2   }{|k|}  \|\rho_0 \|_{L^{p+1}(0,1)}  + 2  M .
\end{align*}
It follows that
\begin{align*}
|W(T)|
 \leq &   \|\rho(\cdot,T)\|_{L^{p+1}(0,1)}\notag\\
 \leq &  \frac{ 2   }{|k|}      \|\rho_0 \|_{L^{p+1}(0,1)}  +   \frac{2}{1-|k|}\sup_{t\in  (0,T)}|d (t)|.
 \end{align*}
 Thus, for $\rho_0$ and $d$ satisfying \eqref{PDE2-31}, i.e.,
 \begin{align*}
  \|\rho_0 \|_{L^{p+1}(0,1)}  +   \sup_{t\in  \mathbb{R}_{\geq 0}}|d (t)|\leq R_0
 \end{align*}
 with some positive constant $R_0$, it holds that
 \begin{align*}
|W(T)|
  \leq &  2R_0\max\left\{\frac{ 1   }{|k|}, \frac{ 1   }{1-|k|} \right\}  .
 \end{align*}
 According to the assumption \textbf{(A2)}, we have
 \begin{align*}
\lambda(W(T))\geq &\lambda(|W(T)|)\notag\\
  \geq &  \Lambda_0:= \lambda\left( 2R_0\max\left\{\frac{ 1   }{|k|}, \frac{ 1   }{1-|k|} \right\}\right).
 \end{align*}
 Then, \eqref{hatV-W(t)} becomes
 \begin{align*}
 {\widehat{V}}( \rho(y,T) ) \leq&  e^{-r \int_0^T\lambda(W(t))\diff t}\widehat{V}(\rho_0(y)) \\
 \leq & e^{-r \int_0^T\Lambda_0 \diff t}\widehat{V}(\rho_0(y))\notag
 \\
 =& e^{-r  \Lambda_0 T}\widehat{V}(\rho_0(y)).\notag
\end{align*}
 The remaining part of the proof is exactly the same as the proof of Proposition~\ref{Prop. PDE2-1}.
\end{proof3}

\section{ISS analysis for  $1$-D wave  equations via GLM}\label{GLF-wave}
In this section, we  employ the GLM to establish  ISS  estimates for  the following  $1$-D wave equation  with  boundary velocity damping:
 \begin{subequations}\label{second hyperbolic PDE}
 \begin{align}
w_{tt}(y,t)=&c^2w_{yy}(y,t) +f(y,t),  (y,t)\in (0,1)\times\mathbb{R}_{>0},\\
 w(0,t)=&0,t\in \mathbb{R}_{>0},\\
 w_y(1,t)=&-kw_t(1,t)+d(t),t\in \mathbb{R}_{>0},\\
 w(y,0)=&w_0(y), y\in(0,1),\\
  w_t(y,0)=&\phi_0(y), y\in(0,1),
 \end{align}
 \end{subequations}
 where $ c$ and $k$ are positive constants satisfying
 \begin{align*}
 ck=1,
 \end{align*}
 $f,d:\mathbb{R}_{>0}\to \mathbb{R}$ are smooth functions, representing distributed in-domain and boundary disturbances, respectively, and  $w_0$ and $ \phi_0$ are smooth functions, representing the initial data.

It is worth mentioning  that the CLM has been employed to establish the ISS in the norm   of  Hilbert spaces for the $1$-D wave equation with different dampings; see \cite{Karafyllis:202231,Karafyllis:2023}. In this paper, we  show how to use the GLM  to establish the ISS in the norm of Banach spaces for the wave equation with  boundary velocity damping.

For any $p\in(1,\infty)$, let $g$ and $G$ be defined by \eqref{Def:G}.
For any $T>0$, let
\begin{align*}
  M:=   \frac{1}{c}\sup_{t\in  (0,T)}|d (t)|.
\end{align*}

For $v\in C(\mathbb{R}_{>0};L^{p+1}(0,1))$, define
\begin{align*}
V_1 (v):= &\int_0^1e^{r  y}G(v(y,t))\diff y,\\
V_2(v):=&\int_0^1e^{-r  y}G(v(y,t))\diff y,
\end{align*}
where $r$ is a   positive constant.
 Let
  \begin{align*}
\widehat{V}(w):= &V_1 (w_t+cw_y -M) +V_1 (-w_t-cw_y -M) +V_2 (w_t-cw_y -M) +V_2 (-w_t+cw_y -M).
\end{align*}

\begin{proposition}The following statements hold  true.
\begin{enumerate}
\item[(i)]For any $p\in(1,\infty)$ and $r\in \mathbb{R}_{>0}$, $\widehat{V}(w)$ is a GISS-LF for system~\eqref{second hyperbolic PDE}  and thus,   system~\eqref{second hyperbolic PDE} is ISS in the   $L^{p+1}$-norm. Furthermore, for any $q\in[2,\infty)$, system~\eqref{second hyperbolic PDE} is ISS    in the $L^{q}$-norm, having the estimate:
 \begin{align}
   \|w_t(\cdot,T)\|_{L^{q}(0,1)}+\|w_y(\cdot,T)\|_{L^{q}(0,1)}
 \leq& 2^{\frac{q-1}{q}} \bigg[   2^{\frac{q+1}{q} } e^{\frac{2r-(cr-\varepsilon)T}{q}}\left(\|\phi_0\|_{L^{q}(0,1)}+\|w_{0y}\|_{L^{q}(0,1)}\right)\notag\\
     &+   \left(\frac{q-1}{\varepsilon}\right) ^{ \frac{q-1}{q}} \left(\frac{8e^{2r}}{cr-\varepsilon}\right)^{\frac{1}{q}} \sup_{(y,t)\in (0,1)\times (0,T)}|f(y,t)| \notag\\
     &
  +   \frac{2}{c}\sup_{t\in  (0,T)}|d (t)| \bigg],\forall T\in \mathbb{R}_{>0},\label{PDE3-ISS-1}
  \end{align}
where  $r,\varepsilon\in\mathbb{R}_{>0}$ are arbitrary constants satisfying $cr-\varepsilon>0$.
\item[(ii)]For any $q\in[2,\infty]$, system~\eqref{second hyperbolic PDE} is ISS    in the $L^{q}$-norm, having the estimate:
\begin{align}
   \|w_t(\cdot,t)\|_{L^{q}(0,1)}+\|w_y(\cdot,t)\|_{L^{q}(0,1)}
 \leq& 8 e^{\frac{4m}{c} }e^{-\frac{1}{2}mt}\left(\|\phi_0\|_{L^{q}(0,1)} +\|w_{0y}\|_{L^{q}(0,1)}\right) \notag\\
 &+  \frac{16}{m }  e^{\frac{4m}{c} } \sup_{(y,s)\in (0,1)\times (0,t)}|f(y,s)|\notag\\
  &+   \frac{4}{c}\sup_{s\in  (0,t)}|d (s)|  ,\forall t\in \mathbb{R}_{>0},\label{PDE3-ISS-2}
  \end{align}
where $m \in\mathbb{R}_{>0}$ is an arbitrary   constant.
  \end{enumerate}
\end{proposition}

\begin{remark} Due to arbitrariness of $r$ in \eqref{PDE3-ISS-1} (or $m$ in \eqref{PDE3-ISS-2}), the disturbance-free system is superstable and hence, finite-time stable, namely, the energy decays to zero in finite time; see \cite{Balakrishnan:1999,Balakrishnan:1996}.
\end{remark}

\begin{proof} We proceed the proof in four steps.

\textbf{Step 1:} Calculate the derivative of $\widehat{V}$.

Indeed, for any $t\in(0,T)$, using integrating by  parts, we obtain
\begin{align}
 \dot{V}_1(w_t+cw_y -M)
=&\int_0^1e^{ r y}g(w_t+cw_y -M)(w_{tt}+cw_{yt})\diff y\notag\\
 =&\int_0^1e^{ r y}g(w_t+cw_y -M)( c^2w_{yy}+f+cw_{yt})\diff y\notag\\
 =&c\int_0^1e^{ r y}g(w_t+cw_y -M)(  w_{t}+c w_{y })_y\diff y+\int_0^1e^{ r y}g(w_t+cw_y -M) f \diff y\notag\\
  =&c\int_0^1e^{ r y}\left(G(w_t+cw_y -M)\right)_y\diff y +\int_0^1e^{ r y}g(w_t+cw_y -M) f \diff y\notag\\
 =&c \left[e^{ r y}G(w_t+cw_y -M)\right]_{y=0}^{y=1} -cr{V}_1(w_t+cw_y -M)\notag\\
 & +\int_0^1e^{ r y}g(w_t+cw_y -M) f \diff y.\label{PDE3-53}
  \end{align}

Using the Young's inequality (see \textbf{(G8)}), we have
\begin{align}
  \int_0^1e^{ r y}g(w_t+cw_y -M) f \diff y
\leq &\varepsilon    \int_0^1e^{ r y}G(w_t+cw_y -M) \diff y\notag\\
 & +\left(\frac{p}{\varepsilon}\right) ^{ p}  \int_0^1e^{ r y}G(|f|) \diff y,\forall  \varepsilon\in \mathbb{R}_{>0}.\label{PDE3-54}
  \end{align}
  We infer from \eqref{PDE3-53} and \eqref{PDE3-54} that
  \begin{align}
 \dot{V}_1(w_t+cw_y -M)
\leq &  -(cr-\varepsilon) {V}_1(w_t+cw_y -M)+\left(\frac{p}{\varepsilon}\right) ^{ p}  V_1(|f|)
  \notag\\
 &  + c \left[e^{ r y}G(w_t+cw_y -M)\right]_{y=0}^{y=1} ,\forall     t\in(0,T).
 \label{PDE3-55}
  \end{align}
  Analogously, we deduce that
  \begin{align}
 \dot{V}_1(-w_t-cw_y -M)
\leq &  -(cr-\varepsilon) {V}_1(-w_t-cw_y -M)+\left(\frac{p}{\varepsilon}\right) ^{ p}  V_1(|f|)\notag\\
 & + c \left[e^{ r y}G(-w_t-cw_y -M)\right]_{y=0}^{y=1} ,\forall    t\in(0,T),
 \label{PDE3-56}
  \end{align}
   \begin{align}
 \dot{V}_2(w_t-cw_y -M)
\leq &  -(cr-\varepsilon) {V}_2(w_t-cw_y -M)+\left(\frac{p}{\varepsilon}\right) ^{ p}  V_2(|f|) \notag\\
 &- c \left[e^{ -r y}G(w_t-cw_y -M)\right]_{y=0}^{y=1} ,\forall    t\in(0,T),
 \label{PDE3-57}
  \end{align}
  and
        \begin{align}
 \dot{V}_2(-w_t+cw_y -M)
\leq &  -(cr-\varepsilon) {V}_2(-w_t+cw_y -M)+\left(\frac{p}{\varepsilon}\right) ^{ p}  V_2(|f|)\notag\\
 & - c \left[e^{ -r y}G(-w_t+cw_y -M)\right]_{y=0}^{y=1} ,\forall    t\in(0,T),
 \label{PDE3-58}
  \end{align}
  where $\varepsilon$ keeps the same value and satisfies
 \begin{align*}
 cr-\varepsilon>0.
 \end{align*}

  Combining \eqref{PDE3-55}, \eqref{PDE3-56}, \eqref{PDE3-57}, and \eqref{PDE3-58}, we obtain
     \begin{align}
 \dot{\widehat{V}}(w)
\leq &  -(cr-\varepsilon) \widehat{V} (w)+2\left(\frac{p}{\varepsilon}\right) ^{ p} \left( V_1(|f|)+V_2(|f|)\right)   +cB(t)
 \notag\\
 \leq &  -(cr-\varepsilon) \widehat{V} (w)+4\left(\frac{p}{\varepsilon}\right) ^{ p}   V_1(|f|)  +cB(t),     \forall    t\in(0,T),
 \label{PDE3-59}
  \end{align}
  where
  \begin{align*}
  B(t):=&
 \left[e^{ r y}G(w_t+cw_y -M)\right]_{y=0}^{y=1}  + \left[e^{ r y}G(-w_t-cw_y -M)\right]_{y=0}^{y=1} -  \left[e^{ -r y}G(w_t-cw_y -M)\right]_{y=0}^{y=1} \notag\\
     &-   \left[e^{ -r y}G(-w_t+cw_y -M)\right]_{y=0}^{y=1} .
  \end{align*}

  \textbf{Step 2:} Handle the boundary term $ B(t)$.

  Using the boundary conditions, the fact that $ck=1$, and the monotonicity of $G$, we have
  \begin{align*}
  B(t)
 =&
 e^{ r }G(w_t(1,t)+cw_y(1,t) -M)    - G(w_t(0,t)+cw_y(0,t) -M) \notag\\
 &+e^{ r }G(-w_t(1,t)-cw_y(1,t) -M)    - G(-w_t(0,t)-cw_y(0,t) -M)\notag\\
 & -  e^{ -r  }G(w_t(1,t)-cw_y(1,t) -M)    + G(w_t(0,t)-cw_y(0,t) -M)
     \notag\\
     &-    e^{ -r  }G(-w_t(1,t)+cw_y(1,t) -M)   +G(-w_t(0,t)+cw_y(0,t) -M)\notag\\
     =& e^{ r }G(w_t(1,t)+c(-kw_t(1,t)+d(t)) -M) - G(w_t(0,t)+cw_y(0,t) -M)\notag\\
& +e^{ r }G(-w_t(1,t)-c(-kw_t(1,t)+d(t)) -M)  - G(-cw_y(0,t) -M)\notag\\
  &-  e^{ -r  }G(w_t(1,t)-c(-kw_t(1,t)+d(t)) -M)  + G(-cw_y(0,t) -M)
     \notag\\
     &-    e^{ -r  }G(-w_t(1,t)+c(-kw_t(1,t)+d(t)) -M)  +G(cw_y(0,t) -M)\notag\\
=&     e^{ r }G( c d(t)  -M)   - G(cw_y(0,t) -M) +e^{ r }G(-c d(t) -M)   - G(-cw_y(0,t) -M) \notag\\
 &+ G(-cw_y(0,t) -M)
      -    e^{ -r  }G(-2w_t(1,t)+c d(t) -M)  +G(cw_y(0,t) -M)\notag\\
=&     e^{ r }G( c d(t)  -M)    +e^{ r }G(-c d(t) -M)    -  e^{ -r  }G(2w_t(1,t)-c d(t) -M) \notag\\
 & -    e^{ -r  }G(-2w_t(1,t)+c d(t) -M) \notag\\
\leq &  e^{ r }G( c d(t)  -M)    +e^{ r }G(-c d(t) -M)   \notag\\
\leq &  e^{ r }G( c |d(t)|  -M)    +e^{ r }G(|-c d(t)| -M)   \notag\\
\leq &  e^{ r }G( M  -M)    +e^{ r }G(M -M)   \notag\\
=& 0,\forall t\in (0,T).
  \end{align*}

  \textbf{Step 3:} Establish the ISS estimate \eqref{PDE3-ISS-1}.

  Since $B(t)\leq 0$ for all $t\in (0,T)$, we deduce from \eqref{PDE3-59} that
     \begin{align}
 \dot{\widehat{V}}(w)
\leq &  -(cr-\varepsilon) \widehat{V} (w) +4\left(\frac{p}{\varepsilon}\right) ^{ p}   V_1(|f|)   \notag\\
\leq & -(cr-\varepsilon) \widehat{V} (w) +4e^{ r} \left(\frac{p}{\varepsilon}\right) ^{ p}\frac{1}{p+1}\|f\|_{\infty,T}^{p+1}  ,
      \forall    t\in(0,T),
  \label{PDE3-61}
  \end{align}
  where, and in the sequel,
  \begin{align*}
  \|f\|_{\infty,T}:=\sup_{(y,t)\in (0,1)\times (0,T)}|f(y,t)|.
  \end{align*}

  Applying the Gronwall's inequality to \eqref{PDE3-61}, we obtain
  \begin{align}
  \widehat{V} (w(y,T))
\leq &  e^{-(cr-\varepsilon)T} \widehat{V} (w(y,0))  + \frac{4e^{ r}}{p+1}\left(\frac{p}{\varepsilon}\right) ^{ p}\|f\|_{\infty,T}^{p+1}\int_0^Te^{-(cr-\varepsilon)(T-t)}   \diff t\notag\\
\leq &e^{-(cr-\varepsilon)T} \widehat{V} (w(y,0)) +\frac{4e^{ r}}{p+1} \left(\frac{p}{\varepsilon}\right) ^{ p} \frac{1}{cr-\varepsilon} \|f\|_{\infty,T}^{p+1}.
  \label{PDE3-62}
  \end{align}
 Using the definition  of  $\widehat{V}$, \textbf{(G7)}, and  \textbf{(G6)}, we infer from \eqref{PDE3-62} that
 \begin{align}
   &\frac{1}{p+1}\left(\|w_t(\cdot,T)\|^{p+1}_{L^{p+1}(0,1)}+\|w_y(\cdot,T)\|^{p+1}_{L^{p+1}(0,1)}\right) \notag\\
 =&
\int_0^1G(|w_t(y,T)|)\diff y+\int_0^1G(|w_y(y,T)|)\diff y\notag\\
 \leq &   e^r\widehat{V} (w(y,T))+2^{p}G(M)+e^r\widehat{V} (w(y,T))+2^{p}G(M)\notag\\
\leq &2 e^{r-(cr-\varepsilon)T} \widehat{V} (w(y,0))  +\frac{8e^{ 2r}}{p+1} \left(\frac{p}{\varepsilon}\right) ^{ p} \frac{1}{cr-\varepsilon} \|f\|_{\infty,T}^{p+1} +2^{p+1}G(M)\notag\\
= &2 e^{r-(cr-\varepsilon)T}    V_1(|\phi_0(y)+cw_{0y}(y)|-M)    
+2 e^{r-(cr-\varepsilon)T} V_2(|\phi_0(y)+cw_{0y}(y)|-M)
   \notag\\
  &+\frac{8e^{ 2r}}{p+1} \left(\frac{p}{\varepsilon}\right) ^{ p} \frac{1}{cr-\varepsilon} \|f\|_{\infty,T}^{p+1} +2^{p+1}G(M)\notag\\
  \leq & 2 e^{r-(cr-\varepsilon)T}    V_1(|\phi_0(y)+cw_{0y}(y)| )  
+2 e^{r-(cr-\varepsilon)T} V_2(|\phi_0(y)+cw_{0y}(y)| )
    \notag\\
 &+\frac{8e^{ 2r}}{p+1} \left(\frac{p}{\varepsilon}\right) ^{ p} \frac{1}{cr-\varepsilon} \|f\|_{\infty,T}^{p+1} +2^{p+1}G(M)\notag\\
  \leq & 2 e^{r-(cr-\varepsilon)T} 2^{p} (  V_1(|\phi_0(y)|)+V_1(|cw_{0y}(y)| )   ) 
+2 e^{r-(cr-\varepsilon)T} 2^{p} (  V_2(|\phi_0(y)|)+V_2(|cw_{0y}(y)| )   )
   \notag\\
  &+\frac{8e^{ 2r}}{p+1} \left(\frac{p}{\varepsilon}\right) ^{ p} \frac{1}{cr-\varepsilon} \|f\|_{\infty,T}^{p+1} +2^{p+1}G(M)\notag\\
  \leq &   2^{p+2}e^{2r-(cr-\varepsilon)T}    \int_0^1 G(|\phi_0(y)| )\diff y +  2^{p+2}c^{p+1}e^{2r-(cr-\varepsilon)T}\int_0^1 G(|w_{0y}(y)| )\diff y    \notag\\
   &+\frac{8e^{ 2r}}{p+1} \left(\frac{p}{\varepsilon}\right) ^{ p} \frac{1}{cr-\varepsilon} \|f\|_{\infty,T}^{p+1}
  +2^{p+1}G(M)\notag\\
  =&   \frac{2^{p+2}}{p+1}e^{2r-(cr-\varepsilon)T}\left(\|\phi_0\|^{p+1}_{L^{p+1}(0,1)}+\|w_{0y}\|^{p+1}_{L^{p+1}(0,1)}\right)
      +\frac{8e^{ 2r}}{p+1} \left(\frac{p}{\varepsilon}\right) ^{ p} \frac{1}{cr-\varepsilon} \|f\|_{\infty,T}^{p+1}
  \notag\\
 &+\frac{2^{p+1}}{p+1}M^{p+1},
  \end{align}
  which implies  that system~\eqref{second hyperbolic PDE} is ISS    in the $L^{p+1}$-norm, having the estimate
  \begin{align}
  \|w_t(\cdot,T)\|_{L^{p+1}(0,1)}+\|w_y(\cdot,T)\|_{L^{p+1}(0,1)}
 \leq& 2^{\frac{p}{p+1}} \bigg[   2^{\frac{p+2}{p+1} } e^{\frac{2r-(cr-\varepsilon)T}{p+1}}\left(\|\phi_0\|_{L^{p+1}(0,1)}+\|w_{0y}\|_{L^{p+1}(0,1)}\right)\notag\\
     &+   \left(\frac{p}{\varepsilon}\right) ^{ \frac{p}{p+1}} \left(\frac{8e^{2r}}{cr-\varepsilon}\right)^{\frac{1}{p+1}} \|f\|_{\infty,T}
  +   2M \bigg],\forall T\in \mathbb{R}_{>0}.\label{PDE3-ISS-1'}
  \end{align}

In addition, due to the arbitrariness of $p$, we conclude that \eqref{PDE3-ISS-1} holds true for all $q\in [2,\infty)$.

  \textbf{Step 4:} Establish the ISS estimate \eqref{PDE3-ISS-2}.

   For any given $m>0$, letting $r=\frac{2m}{c}p$ and $\varepsilon=mp$  in \eqref{PDE3-ISS-1'} yields
    \begin{align*}
     &\|w_t(\cdot,T)\|_{L^{p+1}(0,1)}+\|w_y(\cdot,T)\|_{L^{p+1}(0,1)}\\
 \leq& 2^{\frac{p}{p+1}} \bigg[   2^{\frac{p+2}{p+1} }  e^{\frac{4m}{c}\frac{p}{p+1}}e^{-\frac{p}{p+1}mT}\left(\|\phi_0\|_{L^{p+1}(0,1)}\right. \left. +\|w_{0y}\|_{L^{p+1}(0,1)}\right)\notag\\
     &+ \left(\frac{1}{m}\right)^{\frac{p}{p+1}}\left( \frac{8}{mp}\right)^{\frac{1}{p+1}} e^{\frac{4m}{c}\frac{p}{p+1}}    \|f\|_{\infty,T}
  +   2M \bigg] \notag\\
  \leq &
 2  \bigg[   4 e^{\frac{4m}{c} }e^{-\frac{1}{2}mT}\left(\|\phi_0\|_{L^{p+1}(0,1)} +\|w_{0y}\|_{L^{p+1}(0,1)}\right)  +  \left(\frac{1}{m}\right)^{\frac{p}{p+1}}\left( \frac{8}{m}\right)^{\frac{1}{p+1}}  e^{\frac{4m}{c} }    \|f\|_{\infty,T}
  +   2M \bigg] \notag\\
  = &8 e^{\frac{4m}{c} }e^{-\frac{1}{2}mT}\left(\|\phi_0\|_{L^{p+1}(0,1)} +\|w_{0y}\|_{L^{p+1}(0,1)}\right)+    \frac{16}{m}  e^{\frac{4m}{c} }    \|f\|_{\infty,T}
  +   4M
  ,\forall T\in \mathbb{R}_{>0},
  \end{align*}
  which, along with the arbitrariness of $p>1$, implies   \eqref{PDE3-ISS-2}.
  \end{proof}
\begin{remark} It is worth noting that when using the GLM for ISS analysis, handling each boundary term in
$
B(t)$ independently would make it difficult to obtain effective estimates. Instead, as shown in Step 2, to eliminate the impact of boundary disturbances, all boundary terms in
$
B(t)$ must be treated as a whole in a unified manner.

\end{remark}

\section{Conclusion}\label{sec:conclusion}

In this paper, within the framework of abstract control systems, we have revisited the concept of (classical) ISS-LF, the (classical) ISS-Lyapunov theorem, as well as the notion of GISS-LF and the GISS-Lyapunov theorem. A distinctive feature of the GISS-LF is its dependence on the external inputs, which generalizes the classical ISS-LF and provides greater flexibility for the ISS analysis of concrete PDEs. In particular, in the presence of  Dirichlet boundary disturbances, the GISS-LF demonstrates superior  efficiency than the classical ISS-LF. We refer to this approach of ISS analysis via the construction of GISS-LFs as the generalized Lyapunov method (GLM). To illustrate the general applicability of GLM, we have shown in detail how to construct GLFs for three  representative examples, namely, an $N$-dimensional parabolic PDE, a first order hyperbolic PDE, and a wave equation with boundary damping, and based on which,  established the corresponding ISS estimates in  general $L^q$ spaces. These examples cover the main features of parabolic, first order hyperbolic, and second order hyperbolic equations, providing a validation of the applicability of the GLM to different classes of PDEs.
   It is worth noting that, due to the diversity of PDEs, constructing a universally  applicable GISS-LF in an explicit form for general PDEs remains an open problem.
    In future work, we will continue to focus on developing GISS-LFs tailored to a wider range of specific PDEs,
  explore the intrinsic connection between GISS-LF and controller design, particularly in the context of observer design and  output feedback control.

  %
%
%
%

 \appendix

 \section{Technical lemmas}

\begin{lemma}[Young's inequality with $\epsilon$, {\cite[p.~622]{Evans:2010}}]\label{Young's inequality} Let $r,q\in(1, \infty)$ satisfy $\frac{1}{r}+\frac{1}{q}=1$. It holds that
\begin{align*}
 ab\leq \epsilon a^r+C(\epsilon)b^q,\forall a,b\in \mathbb{R}_{\geq 0},\forall \epsilon\in \mathbb{R}_{>0},
\end{align*}
where $C(\varepsilon):=  (\epsilon r)^{-\frac{q}{r}} q^{-1} $.
\end{lemma}

\begin{lemma}[Gronwall's inequality, {\cite[Page 708]{Evans:2010}}]  \label{Gronwall}
For $T\in\mathbb{R}_{>0}$, let $\phi$ and $\psi$ be integrable functions on $[0,T]$, and  $\eta$  {an absolutely} continuous function on $[0,T]$. If $\eta$ satisfies {the} differential inequality for all most every $t\in [0,T]$:
\begin{align*}
\eta'(t)\leq \phi (t)\eta(t)+\psi(t),
\end{align*}
  then, it holds that
\begin{align*}
\eta(t)\leq e^{\int_{0}^t\phi(s)\text{d}s} \eta(0)+\int_{0}^t  e^{ \int_{s}^t\phi(r)\text{d}r}\psi(s)\diff s , \forall t\in [0,T].
\end{align*}
\end{lemma}

\section{Proofs of properties \textbf{(G1)-\textbf{(G8)}}}\label{Appendix: G}
Indeed, properties \textbf{(G1)}-\textbf{(G3)} can be directly verified via the definition of $g$ and $G$.

Using the increasing monotonicity of $G$, i.e., \textbf{(G3)},  we get $G(s)\leq G(s+\tau)$ for $\tau\geq 0$ and $G(s)\leq G(s-\tau)$ for $\tau\leq 0$, respectively. Thus, \textbf{(G4)} holds true.

Since $G(|s|+\tau)= G(s+\tau)$ for $s\geq 0$ and $G(|s|+\tau)= G(-s+\tau)$ for $s\leq 0$, we obtain \textbf{(G5)}.

 For \textbf{(G6)}, it suffices to note that    $(s+\tau)^{p+1}\leq 2^{p}(s^{p+1}+\tau^{p+1})$  for   $s,\tau\in \mathbb{R}_{\geq 0}$.

By subsequently using  \textbf{(G6)}, \textbf{(G5)}, and \textbf{(G4)}, we have
\begin{align*}G(|s|)\leq &2^pG(|s|-M)+2^pG(M)\notag\\
= &2^p(G(s-M)+G(-s-M))+2^pG(M)\notag\\
\leq & 2^p (G(s+\tau-M)+G(s-\tau-M) +G(-s+\tau-M)+G(-s-\tau-M)) \notag\\
 &+2^pG(M) ,\forall s,\tau,M\in\mathbb{R}.
\end{align*}
Therefore, \textbf{(G7)} holds true.

For \textbf{(G8)}, it suffices to prove the result for positive $s,\tau$. Indeed, for $s,\tau\in\mathbb{R}_{>0} $,
setting $a:=g(s)=s^{p}$, $b:=\tau$, $r:=\frac{p+1}{p}$, $q:=p+1$, $\epsilon :=\frac{\varepsilon}{p+1} $ in Lemma~\ref{Young's inequality}  and using \textbf{(G2)}, we obtain
\begin{align*}
 g(s)\tau\leq &\frac{1}{p+1}\varepsilon (s^{p})^{\frac{p+1}{p}}+C(\epsilon)\tau^{p+1}\notag\\
 =&  \frac{\varepsilon}{p+1}s^{p+1} +\left(\frac{\varepsilon}{p+1}  \frac{p+1}{p}\right)^{-\frac{p+1}{\ \frac{p+1}{p}\ }} \frac{1}{p+1}\tau^{p+1}\notag\\
 =&  \varepsilon  G(s)+\left(\frac{p}{\varepsilon}\right) ^{ p}G(\tau) .
 \end{align*}
Therefore, \textbf{(G8)} holds true.

\vspace{2cm}

\begin{tabular}{ll}
Jun Zheng & Guchuan Zhu \\
School of Mathematics & Department of Electrical Engineering \\
Southwest Jiaotong University &  Polytechnique
                Montr\'{e}al\\
 Chengdu,   611756,  China&  
                Montreal,  H3T 1J4, Canada\\
{\tt zhengjun2014@aliyun.com} & {\tt guchuan.zhu@polymtl.ca}
\end{tabular}

\end{document}